\RequirePackage{fix-cm}
\documentclass[smallextended]{svjour3}       
\smartqed  
\usepackage{graphicx}
\usepackage{amsmath}
\usepackage{multirow}
\usepackage{color}
\usepackage{comment}

\journalname{Operational Research}

\begin{document}

\title{Decomposition Algorithm for the Multi-Trip Single Vehicle Routing Problem with AND-type Precedence Constraints}

\author{
Mina Roohnavazfar \and \\
Seyed Hamid Reza Pasandideh
}

\institute{
	Mina Roohnavazfar \at
    Department of Industrial Engineering, Kharazmi University - Tehran (Iran) and\\
	Department of Control and Computer Engineering, Politecnico di Torino - Turin (Italy) \\
	\email{mina.roohnavazfar@polito.it}
	\and
	Seyed Hamid Reza Pasandideh \at
	Department of Industrial Engineering, Faculty of Engineering, Kharazmi University - Tehran (Iran)\\
	\email{shr\_pasandideh@khu.ac.ir} \textbf{(corresponding author)}
}

\date{Received: date / Accepted: date}

\maketitle

\begin{abstract}
This paper addresses a new variant of the multi-trip single vehicle routing problem where the nodes are related to each other through AND-type precedence constraints. The problem aims at determining a sequence of trips to visit all the nodes respecting every precedence constraint within and among the routes so to minimize the total traveling cost. Our motivation comes from routing problems where a node may have a set of predecessors (not just single one proposed in the dial-a-ride or pickup and delivery problems) resulting in a set of pairwise relations that specify which customers need to be visited before which other ones. We develop three Mixed Integer Programming (MIP) models to formulate the problem. The models are experimentally compared to determine the best one. Moreover, a solution approach based on the Logic-Based Benders Decomposition (LBBD) algorithm is developed which allows to decompose the original problem into an assignment master problem and independent sequencing subproblems. A new valid optimality cut is devised to achieve faster convergence. The cut performance is experimentally investigated by comparing with a recently proposed one in the literature. We further relax the algorithm to find the sub-optimal solution and demonstrate its efficiency. Extensive computational experiments are conducted to assess the proposed algorithms in terms of solution quality and CPU time.

\keywords{Capacitated single vehicle routing problem \and AND-Type Precedence constraints \and Multiple trips \and Logic based benders decomposition algorithm}
\end{abstract}

\section{Introduction}\label{intro}
The Capacitated Vehicle Routing Problem (CVRP) has been a subject of research for many years due to economic importance of the problem, and also the methodological challenges it poses. The classical version of CVRP seeks an optimal nodes assignment to vehicles and optimizes the sequence of nodes within each route. To better represent real-world problems, additional attributes need to be considered leading to different variants of the CVRP.

There exists vast literature on the Multi-Trip Vehicle Routing Problem (MTVRP). In some of them, the fleet is composed of one single vehicle with limited capacity allowed to perform multiple trips, like  helicopters. An important reason for this gain of interest can be of enormous economic importance to some
companies. The standard VRP assumes single use of vehicles over a planning period. However, in several contexts, like home delivery of fresh good, the duration of the routes is short, therefore they must be combined to form a complete workday. In applications where the vehicle capacity is small (electrical vehicles or small-sized vans) or the planning period is large, performing more than one route per
vehicle may be the only practical solution. The MTVRP has become increasingly important, also, due to the advent of electronic services, like e-groceries, where customers can order goods online and have them delivered at home. In such applications, the vehicle is allowed for multiple trips to avoid an oversized fleet.

In many real industrial or social service environments where a set of tasks or services are
performed, partial orders known as Precedence Constraints (PCs) are defined  that specify which nodes need to be visited before which other ones. The PCs play an essential role in a wide range of applications in various fields, such as the assembly industry, distribution of services, products, or passengers, construction projects, production scheduling, and maintenance support. In general cases, precedence constraints are represented by an arbitrary directed acyclic graph in which each arc corresponds to a precedence constraint between a pair of nodes.

In logistics, it is also tough to suppose that the customer nodes are visited independent of each other. Some customers/target locations may have priorities over the others to be served/visited
due to their interconnections as in the Dial-A-Ride or Pickup and Delivery problems. In these problems, each node is visited at most after just one predecessor. For instance, in the Dial-A-Ride problem, for any backhaul node $j$, there is a particular inhaul node $i$ where the relation $i<j$ (we refer to it as conventional precedence constraint) must be met within a route. In practice,
there are cases in which the conventional PCs are not suitable. So, alternative definitions, such as AND precedence, OR precedence, and S-precedence are defined. AND-type precedence constraints are defined when a node has multiple predecessors which have been met before. For example, given node $i$, a set of predecessors $\{j,k,l\}$ is considered which, consequently, result in a set of pairwise relations $\{(j<i),(k<i),(l<i)\}$.

A practical applications of AND-type precedence constraints can be visiting patients by medical personal during the evacuation of a special needs population in anticipation of a disaster. In such a situation, some patients have priorities (due to their urgency or importance) to be visited before some others before they can be transported to the shelters. Another application might be freight transportation by helicopter or ship where some locations need to be reached after some others.

Let's consider a delivery problem in which the order of a customer includes various items collected from different locations. In such a situation the customer must be visited in a route only after (not necessarily immediately) visiting the locations of requested goods. This situation can also be seen in order picker routing problem where a picker walks or drives through the warehouse to collect the requested items and put them in a roll container considering fragility restrictions, stackability, shape, size and weight. For example, to prevent damage to light items, pickers are not
allowed to put heavy items on top of light items. Such physical features and also preferred loading or unloading sequences (to avoid extra effort on sorting and packing the collected items at the end of the retrieving process) can be represented as AND-type PCs that specify which items should be collected before a given one within a trip or among them.

To the best of our knowledge, there is no available paper in the literature of VRP in which the nodes are visited satisfying AND-type PCs along and among the trips. In this research, we describe the Multi-Trip Capacitated Single Vehicle Routing Problem with AND-type Precedence Constraints where the main objective is to find the optimal sequence of routes with the minimum total travelling cost. The nodes are related with AND-type precedence relations which must be respected not only within a route but also among them by determining the order of trips. To address the proposed problem, three mathematical formulations are developed. The performance of the models are experimentally investigated on a set of generated instances. Moreover, a decomposition algorithm based on the Logic-Based Benders Decomposition (LBBD) approach is developed which partitions the original problem into an assignment master problem and sequencing sub-problems. A new extension of a valid optimality cut is developed, and its performance is experimentally investigated. The cuts are generated from the solutions of subproblems and added to the master problem to achieve faster convergence. We also propose a relaxed version of LBBD which allows the algorithm to find feasible solution in less CPU time and even for larger instances.

The contribution of this study is threefold: (i) considering AND-type precedence constraints in the multi-trip capacitated single vehicle routing problem; (ii) three mixed integer programming models are developed and experimentally compared on a set of generated instances; (iii) a logic-based benders decomposition algorithm with a new valid optimality cut and its relaxed version are designed and implemented capable of obtaining good solutions in terms of both quality and computational time.

This paper is organized as follows. Section \ref{Literature} is dedicated to the literature review of the problem. In Section \ref{description}, the proposed problem is described and three mathematical formulations are provided. Moreover, the computational results of the proposed models in solving a set of instances are represented in this part. In Section \ref{LBBD}, we develop the LBBD algorithm to solve the problem. In Section \ref{Computational experiments}, the computational results of the proposed algorithms are provided. Finally, the paper ends with conclusions and some future research suggestions in Section \ref{conclusion}.

\section{Literature review} \label{Literature}
The literature on the proposed problem can be distinguished into three main research areas. The first one deals with the precedence constraints with the main focus on the routing problem. The second part discusses researches on the multi-trip single vehicle routing problem, while the last one includes the studies proposing the logic based benders decomposition algorithm.

\subsection{Problems with Precedence Constraints}\label{Lit PC}
In many real operational research problems where a set of tasks or services are performed, it is tough to suppose that the tasks are independent of each other. In principle, the relative order between a couple of tasks is represented as a pairwise relation named Precedence Constraint (PC). These relations make the problems have a more comprehensive range of applications in various fields. Given nodes $i$ and $j$, precedence constraint $i<j$, referred to a conventional precedence constraint, can be interpreted as activity $i$ is performed before $j$. In practice different precedence constraints are defined such as 'AND', 'OR', 'S' types. The AND-type precedence constraint means that a node can be reached only after visiting all of its predecessors. The OR-type precedence constraint is defined when a node has multiple predecessors, and it can be reached when at least one of its predecessors has been met before. Under the S-type PC, a task can be started once all of its predecessors have been started. So, the task does not need to wait for the completion of its predecessors.

In the context of machine scheduling, Soft Precedence Constraints (SPC) have been defined by a bipartite network at a transshipment port \cite{Zhang et al. 2020}. The SPCs can be violated, but with a certain penalty.
Goldwasser and Motwani \cite{Goldwasser and Motwani 1999} derive inapproximability results for a specific single-machine scheduling problem with AND/OR precedence constraints.
Gillies and Liu \cite{Gillies and Liu 1995} addressed single and parallel machine scheduling problems to meet deadlines considering different structures of AND precedence constraints. They proved NP-completeness of finding feasible schedules in many polynomially solvable settings with only AND-type precedence constraints. Moreover, they give priority-driven heuristic algorithms to minimize the completion time on a multiprocessor.
Mohring et al. \cite{Mohring et al. 2004} provided some algorithms for the more general and complex model of AND/OR precedence constraints. They showed that feasibility and questions related to generalized transitivity could be solved using the same linear-time algorithm. Moreover, they discussed a natural generalization of AND/OR precedence constraints and prove that the same problems become NP-complete in this setting.
Lee et al. \cite{Lee et al. 2012} focused on flexible job-shop scheduling problems with AND/OR precedence constraints in the operations. They provided a MILP model able to find optimal solutions for small-sized instances. They also developed a heuristic algorithm that results in a good solution for the problem regardless of its size. Moreover, a schedule builder who always gives a feasible solution and genetic and tabu search algorithms based on the proposed schedule builder were presented.
Van Den Akker et al. \cite{Van Den Akker et al. 2005} developed a solution framework which provides feasible schedules to minimize the minimax type on a set of identical parallel machines subject to release dates, deadlines, AND/OR precedence constraints. They determined a high quality lower bound by applying column generation to the LP-relaxation.
Agnetis et al. addressed some special cases of job shop and flow shop scheduling problems with S-type precedence constraints \cite{Agnetis et al. 2019}. They provided polynomial exact algorithms for a two-machine job shop and a two-machine flow shop, and an m-machine flow shop with two jobs.

In the context of routing problems, Moon et al. \cite{Moon et al. 2002} addressed the Traveling Salesman Problem with Precedence Constraints (TSPPC) defined as pair-wise relations between each couple of nodes. The PCs are represented as an order under which the nodes must be visited. A genetic algorithm that involves a topological sort and a new crossover operation is proposed to solve the model. Mingozzi et al. \cite{Mingozzi et al.1997} dealt with the TSP with time windows and precedence constraints using a dynamic programming approach. Fagerholt and Christiansen \cite{Fagerholt and Christiansen 2000} considered a TSPPC with a time window to solve the bulk ship scheduling problem. The model is solved as a shortest path problem on a graph. Renaud et al. \cite{Renaud et al. 2000} proposed a heuristic model to solve the pickup and delivery TSP formulated as the TSPPC. Bredstrom and Ronnqvist \cite{Bredstrom and Ronnqvist 2008} developed a mathematical model for the combined vehicle routing and scheduling problem with time windows. The sets of pairwise synchronization and precedence constraints are considered between customer visits, independently of the vehicles. Also, they described some real-world problems to emphasize the importance of the mentioned constraints, such as homecare staff scheduling, airline scheduling, and forest operations. Bockenhauer et al. \cite{Bockenhauer et al. 2013} studied a variant of TSP in which a given subset of nodes are visited in a prescribed order in the computed Hamiltonian cycle. They presented a polynomial-time algorithm to solve the problem. Haddadene et al. \cite{Haddadene et al. 2016} modeled a home health care structure as a variant of vehicle routing problem with time windows and timing constraints. Some patients ask for more than one visit simultaneously or in given priority order. A MILP model, a greedy heuristic, two local search strategies, and three metaheuristics are proposed to solve the problem. Recently, the task assignment problem for a team of heterogeneous vehicles has been investigated in which packages are delivered to a set of dispersed customers subject to precedence constraints. Using graph theory, a lower bound on the optimal time is constructed. Integrating with a topological sorting technique, several heuristic algorithms are developed to solve it \cite{Bai et al. 2019}.

A  closely related case of the VRP with PCs is the Dial-A-Ride problem, which is an exhaustively studied problem. In this problem, the pair-wise relations are inherently represented between pickup and delivery points within a route, i.e., for any backhaul node $j$, there is a particular inhaul node $i$ where the PC  $(i<j)$ must be met within a route. For a comprehensive survey of the developed models, applications, and algorithms which address the Dial-A-Ride problems, the reader is referred to \cite{Ho et al. 2018}, \cite{Molenbruch et al. 2017}, and \cite{Cordeau and Laporte 2007}.

The order-picking problem is one of the main applications of AND-type PCs in the context of routing problems. However, little works in order-picking problems have focused on PCs. Zulj et al. \cite{Zulj et al. 2018} considered the PCs in a warehouse of a German manufacturer of household products, where heavy items are not allowed to be stored on top of delicate items to prevent damage to the delicate items. To avoid the sorting effort at the end of the order-picking process, they propose a picker-routing strategy respecting the precedence constraints. An exact algorithm based on dynamic programming is used to evaluate the strategy and compared with the simple s-shape routing strategy. Dekker et al. \cite{Dekker et al. 2004} investigated combinations of storage assignment strategies and routing heuristics for a real case arising in a warehouse of a wholesaler of tools and garden equipment. A guideline has to be considered indicating that fragile products have to be picked last. Matusiak et al. \cite{Matusiak et al. 2014} presented a simulated annealing method to address the joint order batching and precedence-constrained picker-routing problem in a warehouse with multiple depots. The shortest path through the warehouse is determined using an exact algorithm.
Chabot et al. \cite{Chabot et al. 2017} introduced the order-picking routing problem underweight, fragility, and category constraints. They propose a capacity-index and two-index vehicle-flow formulations as well as four heuristics to solve the problem. Furthermore, a branch-and-cut algorithm is applied to solve the two mathematical models.

It should be noted that the proposed precedence relations in most picker routing problems (mentioned above) have been represented as a pre-specified sequence of nodes. However, in our proposed problem, AND-type PCs are defined under which the nodes assignments to trips and the sequence of nodes in each route are determined in the model.

\subsection{Multi-Trip Single Vehicle Routing Problem}\label{Lit MTS}
In the Multi-Trip Vehicle Routing Problem (MTVRP), each vehicle is able to perform several trips during the planning horizon. In situations where the vehicle capacity is small or applications, such as home delivery of perishable goods like food, the duration of the routes is short, therefore, the vehicles can travel several trips to complete a workday. In 1990, this problem was introduced by Fleischmann (FLE 1990). A survey on this problem can be found in \cite{Sen2018} and \cite{Cattaruzz2016}, \cite{Cattaruzz2018} which give a unified view on mathematical formulations, exact and heuristic approaches, and variants of MTVRP.

In the literature, there are some papers that have been studied multi-trip single vehicle routing problems. Following, we refer to a few ones. Martinez-Salazar et al. \cite{Martinez-Salazar et al. 2014} studied a customer-centric routing problem with multiple trips of a single vehicle with the goal of minimizing the total waiting time of customers (latency). They proposed two mixed integer formulations as well as a metaheuristic algorithm capable of obtaining good solutions in terms of quality and computational time. Azi et al. \cite{Azi et al. 2007} addressed a multi-trip single vehicle routing considering time windows and deadline constraints. Taking into account that the time windows constraints may prevent serving all customers, the objective was maximization of the number of served customers and minimization of total distance. They proposed a method based on an elementary shortest path algorithm with resource constraints. Rivera et al. \cite{Rivera et al. 2016} worked on the multi-trip cumulative capacitated single vehicle routing problem inspired by disaster logistics. The objective is minimization of total arrival time. Two mixed integer linear programming models, a low-based formulation and a set partitioning problem, are proposed for small instances, while for the larger instances an exact algorithm based on resource-constrained shortest path problem is developed. Angel-Bello et al. \cite{Angel-Bello et al. 2013} considered a routing problem with multiple use of a single vehicle and service time in demand points with the aim of minimizing the sum of clients waiting time to receive service. They present and compare two mixed integer formulations for this problem, based on a multi–level network.

\subsection{Logic Based Benders Decomposition algorithm}\label{Lit benders}
The LBBD has mostly aroused interest among the researchers in the planning and scheduling field. However, few papers in vehicle routing area have focused on this approach. Here, we can imply to the some of the most cited studies which uses this approach in various optimization problems. Implementation of the LBBD on parallel machine scheduling problem with sequence-dependent setup times and job availability intervals has yielded in outstanding results in comparison with integer programming (IP) and constraint programming (CP) as proposed in \cite{Gedik et al.2016}. A three-level LBBD solves the outpatient scheduling problem efficiently which involves planning and scheduling decisions in \cite{Riise et al.2016}. Multi-distributed operating room scheduling (DORS) problem is addressed using LBBD proposed by Roshanaei et al. \cite{Roshanaei et al.2017}. The algorithm was enhanced using effective acceleration cuts led to faster convergence. Their proposed problem tackles with allocation of patients as well as scheduling of operating rooms. The study performed by Barzanji et al. \cite{Barzanji et al.2020} showed that the proposed LBBD algorithm has outstanding results over the existing heuristic-based search methods for the type-1 of integrated process planning and scheduling (IPPS) problem named flexible job shop scheduling with process plan flexibility problem. Simultaneous planning, lot sizing and scheduling problem is studied by Martinez-Salazar et al. \cite{Martinez-Salazar et al. 2014}. They proposed an integrated branch-and-check with MIP heuristic based on the logic-based Benders platform. The LBBD is deployed effectively in the export dry bulk terminals problem which involves two allocation problems as planning and scheduling parts (see, \cite{Unsal2019}).
Recently, the LBBD was applied to address the heterogeneous fixed fleet vehicle routing problem with time windows with the aim of cost minimization \cite{Faganello Fachini and Armentano 2020}. Valid optimality and feasibility cuts were devised to guarantee the convergence of the algorithm. Extensive computational experiments illustrated the effectiveness of the suggested approach.

\subsection{Summarizing reviewed literature}\label{Lit summary}
In this section, we provide Table \ref{table-literature} to summarize the related literature. As shown, no available research in the literature of vehicle routing and even picker routing problems considers AND-type PCs.

\begin{table}
\caption{Summarized literature reviewed in this section}
\label{table-literature}
\scalebox{0.7}{
\begin{tabular}{llll}
\hline\noalign{\smallskip}
\textbf{Publication} & \textbf{Problem}  & \textbf{Precedence Constraints} & \textbf{Solution Approach}\\
\noalign{\smallskip}\hline\noalign{\smallskip}
\cite{Rivera et al. 2016}&  multi-trip single vehicle routing  & -  & exact algorithm \\
 \cite{Azi et al. 2007}   &  multi-trip single vehicle routing  & -  & exact algorithm \\
 \cite{Angel-Bello et al. 2013} &  multi-trip single vehicle routing  & -  & Cplex solver \\
 \cite{Martinez-Salazar et al. 2014} & multi-trip single vehicle routing & -  & meta-heuristic algorithm \\
 \cite{Van Den Akker et al. 2005} & parallel machine scheduling & AND/OR type PCs & column generation \\
 \cite{Zhang et al. 2020}& machine scheduling & Soft type PCs & approximation algorithms \\
 \cite{Lee et al. 2012}&  flexible job-shop scheduling & AND/OR type PCs &heuristic algorithm \\
 \cite{Gillies and Liu 1995}&   machine scheduling & AND type PCs  &heuristic algorithms \\
 \cite{Agnetis et al. 2019} &  machine scheduling  &  S type PCs & exact algorithms\\
 \cite{Goldwasser and Motwani 1999}& assembly sequencing & AND/OR type PCs & - \\
 \cite{Moon et al. 2002} & travelling salesman problem & conventional PCs & genetic algorithm \\
 \cite{Mingozzi et al.1997} & travelling salesman problem & conventional PCs    & dynamic programming \\
 \cite{Renaud et al. 2000} & pickup and delivery TSP & conventional PCs  & heuristic model \\
 \cite{Bredstrom and Ronnqvist 2008} & combined vehicle routing and scheduling&  conventional PCs  & -\\
 \cite{Bockenhauer et al. 2013} & travelling salesman problem  & pre-specified sequence & heuristic \\
 \cite{Haddadene et al. 2016} & vehicle routing problem  & given priority order & meta-heuristics and heuristic \\
 \cite{Bai et al. 2019}  & assignment problem & conventional PCs   &  heuristic \\
 \cite{Zulj et al. 2018} & order-picking  & pre-specified sequence & dynamic programming \\
 \cite{Dekker et al. 2004} & storage assignment and routing   & pre-specified sequence  & heuristics \\
 \cite{Matusiak et al. 2014}  & order batching and picker-routing & pre-specified sequence & meta-heuristic \\
 \cite{Chabot et al. 2017} &  order-picking routing & pre-specified sequence &branch-and-cut and heuristics \\
 \cite{Gedik et al.2016}  & Machine scheduling & - & LBBD \\
 \cite{Riise et al.2016}  & multi-modal outpatient scheduling & - & LBBD \\
 \cite{Roshanaei et al.2017} & operating room scheduling  &  - & LBBD \\
 \cite{Barzanji et al.2020}  & integrated planning and scheduling  & - & LBBD \\
 \cite{Martinez-Salazar et al. 2014} & planning, lot sizing and scheduling & - & LBBD \\
 \cite{Unsal2019}            & planning and scheduling & -  & LBBD \\
 \cite{Faganello Fachini and Armentano 2020} & vehicle routing & - & LBBD \\
  \textbf{current research} &  \textbf{multi-trip single vehicle routing}  & \textbf{And type PCs}  & \textbf{LBBD} \\
\noalign{\smallskip}\hline
\end{tabular}}
\end{table}

As mentioned before, the most related research in logistics includes dial-a-ride or pickup and delivery problems which only focuses on the pairwise relations between a couple of nodes (inhaul and backhaul nodes) being satisfied within a route. However, we want to stress the fact that assuming such relations, under which at most one predecessor per node is defined which must be visited in the same route, can be restrictive in many real-life situations. In this regard, we propose the AND-type PCs represented on a directed acyclic graph, where given a node, a set of pairwise relations are met within and among the trips. Thus, our work contributes to the literature by describing, modeling and providing solution approaches for Multi-Trip Capacitated Single Vehicle Routing Problem with AND-type PCs.

\section{Problem Description and mathematical formulation} \label{description}
The proposed problem is a generalization of the CVRP in which the nodes are related to each other using AND-type precedence constraints, and the fleet is composed of a single vehicle capable of travelling multiple trips per working day. The problem aims at determining a sequence of trips required to visit all the nodes, assigning nodes to the trips, and finding the nodes sequence in each trip so to minimize the total travelling cost.

The proposed problem is defined on a directed acyclic graph $G=(N,A,c)$, where  $N=\{0,1,2,…,n\}$ is a set of geographically located nodes including the depot (node $0$), and $A=\{(i,j):i,j \in N, i\neq j\}$ is the set of arcs. For each arc $(i,j)$, a travel cost $c_{ij}$ is associated. The cost matrix is symmetric, i.e., $c_{ij}=c_{ji}, \forall i,j \in V, i\neq j$. Given node $i$, set $AND_i$ is defined including the AND-type predecessors of node $i$. The proposed PCs are respected between nodes within and among the trips. Given PC $i<j$, if node $i$ and $j$ are, respectively, assigned to trips $s$ and $k$, then trip $s$ is done before starting trip $k$ by the vehicle. In such a way, the PC $i<j$ is met among the trips. As an example, let’s consider the nodes and the AND-type precedence constraints as depicted in Figure \ref{fig-example}. It can be seen that the vehicle travels three trips to visit all the nodes. According to the defined PCs, the set of AND-type predecessors of node $r$ includes $\{g,c\}$. As depicted, both predecessors are visited before node $r$ in the first trip of vehicle. Concerning AND predecessors of node $a$, it can be noticed that both nodes $k$ and $m$ are visited before node $a$ in such a way that node $k$ is assigned to the same trip as node $a$, while node $m$ has been visited in previous trip. Finally, two nodes $h$ and $p$ are not allocated to the same trip as node $u$. However, the related PCS are met while both nodes are visited before node $u$ in the previous trip.

\begin{figure}
  \includegraphics[width=0.6\textwidth]{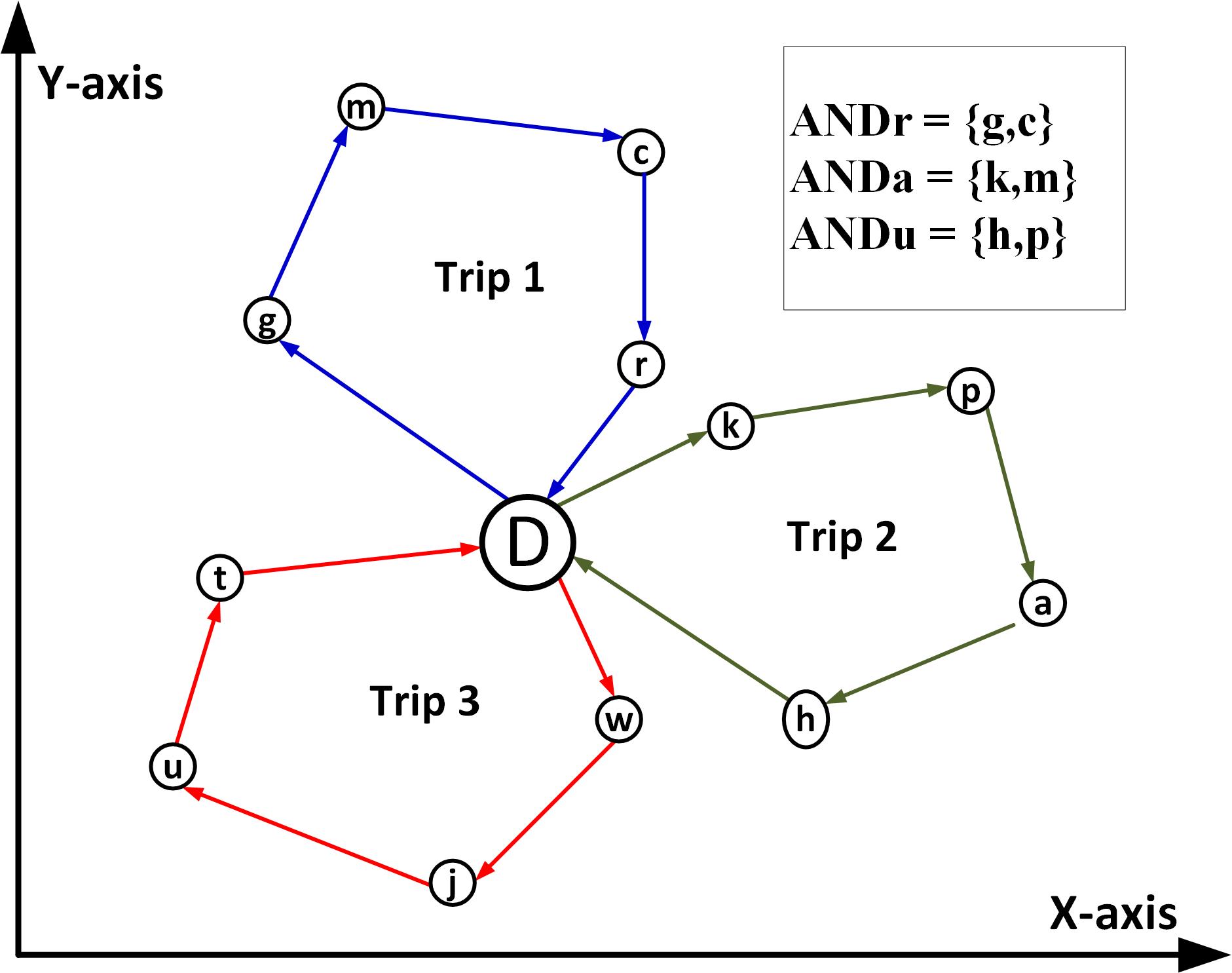}
\caption{Illustration of an example with AND-type PCs respecting within and among the trips} \label{fig-example}
\end{figure}

Let’s assume the following restrictions:
\begin{itemize}
\item Each trip starts and ends at the depot;
\item Each node is visited once;
\item Total demand of each route does not exceed the capacity of the vehicle;
\item Number of trips cannot exceed the Maximum value $p$;
\item PCs are met between nodes within and among the trips.
\end{itemize}

In the next subsections, three mathematical formulations are developed to address the proposed problem. Let us introduce the following notations and parameters:
\begin{itemize}
\item $N=\{0,1,2,...,n\}$: set of nodes including depot;
\item $\acute{N}=\{0,1,...,n,n+1,...,n+p\}$: set of nodes including $p$ dummy depots;
\item $c_{ij}$: travelling cost from node $i$ to node $j$;
\item $p$: upper bound for the number of trips;
\item Q: capacity of the vehicle;
\item $d_i$: demand of node $i$;
\item $M$: a big scalar;
\item $AND_i$: set of predecessors of node $i \in N \setminus\{0\}$.
\end{itemize}

\subsection{Three-index MIP model} \label{3-index}
In this section, the Three-index MIP model of the proposed problem is represented. Let us introduce following decision variables:\\

$x_{ijr}$: binary variable takes value $1$ if arc $(i,j)$ is linked in trip $r$, $0$ otherwise;

$s_{ir}$: integer variable indicating the position of node $i$ in trip $r$;

$u_r$: binary variable takes value $1$ if the vehicle travels trip $r$, $0$ otherwise.\\

Then, the model can be stated as:
\begin{equation}\label{obj-3index}
  Z1= \min \sum_{r=1}^{p} \mathop{{\sum_{j=0}^{n}} \sum_{i=0}^{n}}_{i\neq j} c_{ij} x_{ijr}
\end{equation}
subject to
\begin{equation} \label{1-3index}
\begin{aligned}
  \sum_{j=1}^{n} x_{0jr}=u_r  &&&&&&  \forall r \in \{1,...,p\},
\end{aligned}
\end{equation}
\begin{equation} \label{2-3index}
\begin{aligned}
  \sum_{i=0,i\neq j}^{n} \sum_{r=1}^{p} x_{ijr}=1  &&&&&&  \forall j \in N \setminus\{0\},
\end{aligned}
\end{equation}
\begin{equation}\label{3-3index}
\begin{aligned}
  \sum_{i=0}^{n}x_{ijr}= \sum_{i=0}^{n} x_{jir}  &&&&&&  \forall j \in N \setminus\{0\}, \forall r \in \{1,...,p\},
\end{aligned}
\end{equation}
\begin{equation}\label{4-3index}
\begin{aligned}
  \mathop{\sum_{i=0}^{n} \sum_{j=1}^{n}}_{i\neq j} d_j x_{ijr}\leq Q .u_r  &&&&&&  \forall r \in \{1,...,p\},
\end{aligned}
\end{equation}
\begin{equation} \label{5-3index}
\begin{aligned}
u_{r} \leq u_{\acute{r}} &&&&&& \forall r,\acute{r} \in \{1,...,p\}_{r> \acute{r}},
\end{aligned}
\end{equation}
\begin{equation}\label{6-3index}
\begin{aligned}
  \sum_{i=0}^{n} \sum_{k=1}^{r} x_{iek} \geq \sum_{i=0}^{n} x_{isr}  &&&&&& \forall e \in AND_s, r \in \{1,...,p\},
\end{aligned}
\end{equation}
\begin{equation}\label{7-3index}
\begin{aligned}
  s_{ir}+ 1 -M (1-x_{ijr})\leq s_{jr}  &&&&&& \forall i\neq j \in N, r \in \{1,...,p\},
\end{aligned}
\end{equation}
\begin{equation}\label{8-3index}
\begin{aligned}
  s_{ir}+ 1 \leq s_{jr}  &&&&&& \forall i \in Pr_j, r \in \{1,...,p\},
\end{aligned}
\end{equation}
\begin{equation}\label{9-3index}
\begin{aligned}
0 \leq s_{ir} \leq \mathop{\sum_{i=0}^{n} \sum_{j=1}^{n}}_{i\neq j} x_{ijr}  &&&&&& \forall i \in N, r \in \{1,...,p\},
\end{aligned}
\end{equation}
\begin{equation}\label{10-3index}
\begin{aligned}
 s_{ir} \geq 0 &&&&&& \forall i \in N, r \in \{1,...,p\},
\end{aligned}
\end{equation}
\begin{equation}\label{11-3index}
\begin{aligned}
 x_{ijr}, u_r \in \{0,1\} &&&&&& \forall i,j \in N, r \in \{1,...,p\},
\end{aligned}
\end{equation}\\
Equation (\ref{obj-3index}) is the objective function which minimizes the total travelling costs. Equations (\ref{1-3index}) state that each trip starts from the depot. Equations (\ref{2-3index}) represent that each node is visited exactly once in just one trip. Constraints (\ref{3-3index}) declare that the vehicle, in each trip, enters and leaves a node exactly once. Vehicle capacity is guaranteed by equations (\ref{4-3index}) in each trip. Using constraints (\ref{5-3index}), the sequence of vehicle trips is determined so that trip $r$ cannot be done unless trip $r-1$ has been performed before. Precedence constraints are met in the nodes allocations to trips as restricted in (\ref{6-3index}). It indicates that every predecessor of node $s$ in trip $r$ are assigned to the same trip or its previous ones. Constraints (\ref{7-3index}), which link the position variable $s_{ir}$ and the binary variable $x_{ijr}$ together, insure the order of nodes in each trip. Constraints (\ref{8-3index}) enforce the precedence constraints among the nodes inside each route. Constraints (\ref{9-3index}) limit the upper bound of the position variables. This constraint along with (\ref{7-3index}) prevent the model from creating subtours. Finally, constrains (\ref{10-3index}) and (\ref{11-3index}) define the integer and binary variables, respectively.

\subsection{Two-index MIP model} \label{2-index}
In this section, the Two-index MIP model for the proposed problem is developed. To assign the nodes to the trips, dummy depots are used as separators between vertices. It means the nodes positioned between each two depots form a trip. Therefore, the size of the set $N=\{0,1,2,...,n\}$ increases by the maximum number of trips which leads to set $\acute{N}=\{0,1,...,n,n+1,...,n+p\}$ with $n+1+p$ nodes. It should be noted that all the travelling costs for these dummy depots are assumed to be the same as the original depot, i.e., $c_{0j}=c_{n+r j},r \in {1,2,...,p}$. In this formulation, to respect the precedence constraints, continuous variables $Po_i, \forall i \in \acute{N}$ are introduced which indicate the nodes positions in the overall sequence of all the nodes (including dummy depots) in the set $\acute{N}$, no matter in which trip a node is visited. The proposed variables and developed Two-index MIP model are represented as follows:\\

$y_{ij}$: binary variable takes value $1$ if arc $(i,j)$ is linked, $0$ otherwise;

$Po_i$: integer variable indicating the position of node $i \in \acute{N}$;

$Cd_i$: integer variable indicating the cumulative demand of the trip at node $i \in \acute{N}$.\\

Objective:
\begin{equation}\label{obj-2index}
  Z2= \min \mathop{\sum_{i=0}^{n+p} \sum_{j=0}^{n+p}}_{i\neq j} c_{ij} y_{ij}
\end{equation}
subject to
\begin{equation} \label{1-2index}
\begin{aligned}
  \sum_{i=0,j\neq i}^{n+p} y_{ij}=1  &&&&&&  \forall j \in \acute{N} \setminus\{0\},
\end{aligned}
\end{equation}
\begin{equation} \label{2-2index}
\begin{aligned}
  \sum_{j=1,j\neq i}^{n+p} y_{ij}=1  &&&&&&  \forall i \in \acute{N}\setminus\{n+p\},
\end{aligned}
\end{equation}
\begin{equation} \label{3-2index}
\begin{aligned}
  Cd_i + d_j -M (1-y_{ij}) \leq Cd_j &&&&&&  \forall i\neq j \in \acute{N},
\end{aligned}
\end{equation}
\begin{equation} \label{4-2index}
\begin{aligned}
  d_i \leq Cd_i \leq Q &&&&&&  \forall i \in \acute{N},
\end{aligned}
\end{equation}
\begin{equation} \label{5-2index}
\begin{aligned}
  Po_j + M (1-y_{ij}) \geq Po_i +1 &&&&&&  \forall i\neq j \in \acute{N},
\end{aligned}
\end{equation}
\begin{equation} \label{6-2index}
\begin{aligned}
  1 \leq Po_i \leq  |\acute{N}|  &&&&&&  \forall i \in \acute{N},
\end{aligned}
\end{equation}
\begin{equation} \label{7-2index}
\begin{aligned}
  Po_j \geq  Po_i +1 &&&&&&  \forall i \in AND_j,
\end{aligned}
\end{equation}
\begin{equation} \label{8-2index}
\begin{aligned}
  Po_i \geq 0 &&&&&&  \forall i \in \acute{N},
\end{aligned}
\end{equation}
\begin{equation} \label{10-2index}
\begin{aligned}
  Cd_i \geq 0 &&&&&&  \forall i \in \acute{N},
\end{aligned}
\end{equation}
\begin{equation} \label{10-2index}
\begin{aligned}
  y_{ij} \in \{0,1\} &&&&&&  \forall i,j \in \acute{N},
\end{aligned}
\end{equation}\\

As usual objective function (\ref{obj-2index}) minimizes the total travelling costs. Equations (\ref{1-2index}) and (\ref{2-2index}) are the regular routing rules. Constraints (\ref{3-2index}) and (\ref{4-2index}), derived form the well-known Miller-Tucker-Zemlin formulation, ensure the routes connectivity and limit the number of nodes in a trip considering the vehicle capacity. Constraints (\ref{5-2index}), which relate the two types of variables, ensure the order of linked nodes in the overall sequence of the total nodes. The position value is bounded by restriction (\ref{6-2index}).
The precedence constraints are satisfied both inside and among the trips simultaneously by imposing constrains (\ref{7-2index}), since the variables have no index associated to the trips. Finally, constrains (\ref{8-2index})-(\ref{10-2index}) define the integer and binary variables, respectively.

\subsection{Integrated assignment and sequencing MIP model} \label{integrated}
In this section, we represent an Integration model containing distinct assignment and sequencing decisions. The variables and formulation are presented as follows:\\

$x_{ir}$: binary variable takes value $1$ if node $i \in N$ is assigned to trip $r$, $0$ otherwise;

$y_{ijr}$: binary variable $1$ if arc$(i,j)$ is linked in trip $r$, $0$ otherwise;

$po_{ir}$: integer variable indicating the position of node $i \in N$ in trip $r$;

$u_r$: binary variable takes value $1$ if the vehicle travels trip $r$, $0$ otherwise.\\

Since each trip starts from and ends to the depot, the first position of each trip is associated to the depot while the last position is dedicated to the node which directly precedes depot at the last linked arc$(i,0)$ in trip $r$.

Then, the model is stated as:
\begin{equation}\label{obj-integrated}
  Z3= \min \sum_{r=1}^{p} \mathop{\sum_{j=0}^{n} \sum_{i=0}^{n}}_{i\neq j} c_{ij} y_{ijr}
\end{equation}
subject to
\begin{equation} \label{1-integrated}
\begin{aligned}
  \sum_{r=1}^{p} x_{ir}=1  &&&&&&  \forall i \in N \setminus\{0\},
\end{aligned}
\end{equation}
\begin{equation} \label{2-integrated}
\begin{aligned}
  \sum_{i=1}^{n} d_i. x_{ir} \leq Q. u_r  &&&&&&  \forall r \in \{1,...,p\},
\end{aligned}
\end{equation}
\begin{equation} \label{3-integrated}
\begin{aligned}
 x_{jr}= \sum_{i=0,i\neq j}^{n} y_{ijr}   &&&&&&  \forall j \in N, r \in \{1,...,p\},
\end{aligned}
\end{equation}
\begin{equation} \label{4-integrated}
\begin{aligned}
 x_{ir}= \sum_{j=0,j\neq i}^{n} y_{ijr}   &&&&&&  \forall i \in N, r \in \{1,...,p\},
\end{aligned}
\end{equation}
\begin{equation} \label{5-integrated}
\begin{aligned}
u_{r} \leq u_{\acute{r}} &&&&&& \forall r,\acute{r} \in \{1,...,p\}_{r> \acute{r}},
\end{aligned}
\end{equation}
\begin{equation} \label{6-integrated}
\begin{aligned}
 \sum_{k=1}^{r} x_{ek}\geq x_{sr}   &&&&&&  \forall e \in AND_s, r \in \{1,...,p\},
\end{aligned}
\end{equation}
\begin{equation}\label{7-integrated}
\begin{aligned}
  po_{jr}+ M (1-y_{ijr})\geq po_{ir}+1  &&&&&& \forall i\neq j \in N, r \in \{1,...,p\},
\end{aligned}
\end{equation}
\begin{equation}\label{8-integrated}
\begin{aligned}
  1\leq po_{ir} \leq \sum_{j=0,j\neq i}^{n} x_{jr} &&&&&& \forall i \in N, r \in \{1,...,p\},
\end{aligned}
\end{equation}
\begin{equation}\label{9-integrated}
\begin{aligned}
  po_{jr} \geq po_{ir} +1-M (2-x_{ir}+x_{jr}) &&&&&& \forall i \in Pr_j, r \in \{1,...,p\},
\end{aligned}
\end{equation}
\begin{equation}\label{10-integrated}
\begin{aligned}
  po_{jr} \geq 0 &&&&&& \forall j \in N, r \in \{1,...,p\},
\end{aligned}
\end{equation}
\begin{equation}\label{11-integrated}
\begin{aligned}
  y_{ijr}, x_{ir}, u_r \in \{0,1\} &&&&&& \forall i,j \in N, r \in \{1,...,p\},
\end{aligned}
\end{equation}\\
The objective function is represented by equation (\ref{obj-integrated}). Equations (\ref{1-integrated}) guarantee that each node is allocated to exactly one trip. The vehicle capacity is ensured in each trip by equations (\ref{2-integrated}). Equations (\ref{3-integrated}) and (\ref{4-integrated}) relate the two assignment and sequencing variables in a way that each node must be entered and exited exactly once in each trip. Using constraints (\ref{5-integrated}), the sequence of vehicle trips is determined so that trip $r$ cannot be done unless trip $r-1$ has been performed before. Precedence constraints are met in the nodes allocation to the trips as restricted in (\ref{6-integrated}). Constraints (\ref{7-integrated}) relate the two types of the variables to ensure the nodes positions for the linked arc in each trip. The position variables are limited by restrictions (\ref{8-integrated}). Equation (\ref{9-integrated}) impose precedence constraints within each trip using assignment and position variables.
Finally, constrains (\ref{10-integrated}) and (\ref{11-integrated}) define the Integer and binary variables, respectively.

In the next subsection, a set of experiments are conducted to evaluate and compare the performance of the three presented models on a set of generated instances.

\subsection{Comparison of the three MIP formulations} \label{models comparison}
In this section, the performance of the three developed models are evaluated through an extended set of small instances proposed by Martinez-Salazar et al. \cite{Martinez-Salazar et al. 2014}. The instances having 10, 15, 20, 25 and 30 nodes are randomly generated from points with real coordinates using a uniform distribution in the range $[0,100]$. Rounded Euclidean distances are taken as travel cost between each pair of nodes. The maximum number of trips is fixed in 2 for 10-nodes instances and 3 for other sizes. Depending on the size of the instance, $Q$ was set to 120,120, 140, 180, 200. Demand $d_i$ is randomly assigned with a value of 10, 20 or 30 in such a way that the sum of the demands
is between $(p + 1)Q$ and $pQ$ to ensure feasibility. There are 5 instances for each value of $n$.

An upper triangular matrix without the diagonal called Precedence Matrix (PM) is developed to represent the AND-type precedence constraints. Each element of PM denotes whether or not a precedence relation exists between the two corresponding nodes. If node $s$ have AND-type predecessors $\{i,j,k\}$, then $PM_{is}=1$, $PM_{js}=1$, $PM_{ks}=1$. In case that there is no PC between the two nodes, the corresponding element of the matrix is zero. The representation of precedence constraints as an upper triangular matrix ensures feasibility of the relations.

To generate the PCs, we somehow adopt the scheme proposed by Derriesel and Monch \cite{Derriesel and Monch 2011} who addressed the parallel machines with sequence-dependent setup times, precedence constraints, and ready times. The precedence relations are inserted using the factor $\tau =\{0,0.4,0.8\}$ to evaluate PCs' impact by considering three different scales. Given a column, if a chosen random number from $U[0,1]$ is less than $\tau$, we do not consider any predecessors for the node associated with that column. Otherwise, the number of predecessors is chosen according to $U[0, n-1^{th}]$, where $n^{th}$ is the number of that column. Then, the predecessors are randomly selected from the set of already generated nodes $\{1^{th},...,n-1^{th}\}$.

As a result, the total number of instances is $75$ which is the combinations of the number of nodes $n \in \{10,15,20,25,30\}$, the scale of precedence constraints $\tau \in \{0,0.4,0.8\}$, and 5 generated instances for each combination of $n$ and $\tau$. All the models and codes are executed by GAMS 24.1 on corei5 pc with 2.50 GHz CPU and 4 GB of RAM, and the time limit is set to 5400 seconds.

Table \ref{table-3models} presents the results of experiments for comparing the three models.
Column 1 displays the number of nodes of the instances, while column 2 indicates the scale of precedence constraints. Minimum, maximum and average CPU time (in seconds) are shown from columns 3 to 5 for three-index model, from columns 8 to 10 for two-index, and from columns 13 to 15 for integrated one. Columns 6, 11 ,and 16 show the number of instances (out of 5) which can be optimally solved (denoted by opt). A dash in these columns means that the model was unable to find optimal solutions for non of the 5 instances in the time limit. Values in columns 7, 12, 17 are the percentage of average optimality gap (denoted by $\varDelta(\%)$) over non-optimized instances.

The results highlights way better performance of the Two-index model than the other two ones in terms of CPU time, the number of optimally solved instances, and the gap in dealing with instances more than 20 nodes. As it can be observed, the three-index and integrated formulations can not reach optimal solution in the time limit 5400 seconds in dealing with instances with more than 25 and 20 nodes, respectively. However, the two-index formulation is able to optimally solved all the instances, while the three-index and integrated models can find optimal solutions for, respectively, 52 and 35 out of 75 instances.

Not surprisingly, by increasing the instances sizes, the percentage of averaged optimality gap for non-optimized instances grows except little discontinuities. However, the different behaviour is noticed as the scale of PCs increases. It seems, having more PCs dose not necessarily lead to the more complexity of the problem, and consequently, growing optimality gap and the number of non-optimized instances.

Despite of the quality obtained by the models, we also want to point out its efficiency. As expected, the CPU time grows as the size of the problem increases for the three models. Moreover, the performance of the three models in dealing with small-sized instances (up to 15 nodes) are similar, while for larger instances, it is clear that the two-index formulation is capable of solving all instances in far less time than the two other ones.

Taking into account the various values for the scale of PCs, it can be seen that the models in dealing with different sizes do not show the exactly similar trend. It seems increasing or decreasing the number of PCs can not, necessarily, affect the problem complexity. In some cases, more PCs leads to less computational time and in some other cases the opposite is observed which means that not the number of PCs but the structure of them on graph may increase or decrease the problem complexity. This topic has also been addressed in \cite{Prot and Bellenguez-Morineau 2018}.

\begin{table}
\caption{Computational results of the three proposed MIP models.}
\label{table-3models}
\scalebox{0.75}{
\begin{tabular}{rrrrrrrrrrrrrrrrr}
\hline\noalign{\smallskip}
\multicolumn {2}{c}{Instance}& \multicolumn {5}{c}{Three-index model}&\multicolumn {5}{c}{Two-index model} & \multicolumn {5}{c}{Integrated model}\\
  \hline \noalign{\smallskip}
  \multirow {2}{*}{n}& \multirow {2}{*}{$\tau$}& \multicolumn {3}{c}{CPU time (s)} & \multirow {2}{*}{opt}& \multirow {2}{*}{$\varDelta(\%)$}&\multicolumn {3}{c}{CPU time (s)} & \multirow {2}{*}{opt}& \multirow {2}{*}{$\varDelta(\%)$}&\multicolumn {3}{c}{CPU time (s)} & \multirow {2}{*}{opt}& \multirow {2}{*}{$\varDelta(\%)$}\\
 \cline{3-5} \cline{8-10}  \cline{13-15}
  \noalign{\smallskip}
  & & Min & Max & Avg & & & Min & Max & Avg & & & Min & Max & Avg & & \\
\noalign{\smallskip}\hline\noalign{\smallskip}
 10 & 0   &  5   &  8   & 7  & 5 & 0  & 5  & 6  & 5    & 5 & 0 & 5  & 7   & 6   & 5 & 0 \\
 10 & 0.4 &  5   &  7   & 6  & 5 & 0  & 5  & 6  & 5    & 5 & 0 & 6  & 7   & 7   & 5 & 0 \\
 10 & 0.8 &  6   &  8   & 7  & 5 & 0  & 5  & 7  & 6    & 5 & 0 & 6  & 8   & 7   & 5 & 0 \\
 15 & 0   &  15  &  28  & 20 & 5 & 0  & 13 & 19 & 15   & 5 & 0 & 16 & 32  & 23  & 5 & 0 \\
 15 & 0.4 &  12  &  32  & 19 & 5 & 0  & 15 & 30 & 22   & 5 & 0 & 17 & 47  & 30  & 5 & 0 \\
 15 & 0.8 &  18  &  36  & 25 & 5 & 0  & 15 & 27 & 19   & 5 & 0 & 20 & 35  & 27  & 5 & 0 \\
 20 & 0   & 449  & 830  & 638& 5 & 0  & 58 & 191& 113  & 5 & 0 & 577& 1005& 849 & 4 & 8 \\
 20 & 0.4 & 315  & 961  & 746& 4 & 12 & 66 & 219& 181  & 5 & 0 & 490& 490 & 190 & 1 & 33 \\
 20 & 0.8 & 388  & 827  & 673& 5 & 0  & 73 & 327& 202  & 5 & 0 & -  & 5400& 5400& - & 28 \\
 25 & 0   & 3800 & 4729 &4201& 3 & 16 & 360& 844& 608  & 5 & 0 & -  & 5400& 5400& - & 19 \\
 25 & 0.4 & 1931 & 3855 &3128& 4 & 22 & 392& 797& 575  & 5 & 0 & -  & 5400& 5400& - & 18 \\
 25 & 0.8 & 3669 & 3669 &3669& 1 & 27 & 466& 991& 831  & 5 & 0 & -  & 5400& 5400& - & 35 \\
 30 & 0   &  -   & 5400 &5400& - & 19 &1377&2894& 2159 & 5 & 0 & -  & 5400& 5400& - & 23 \\
 30 & 0.4 &  -   & 5400 &5400& - & 35 &1026&2603& 1844 & 5 & 0 & -  & 5400& 5400& - & 29 \\
 30 & 0.8 &  -   & 5400 &5400& - & 30 &1369&3082& 2762 & 5 & 0 & -  & 5400& 5400& - & 36 \\
\noalign{\smallskip}\hline
\end{tabular}}
\end{table}

\section{Logic Based Benders Decomposition algorithm} \label{LBBD}
In this paper, the Logic Based Benders Decomposition algorithm is implemented to address the original problem using a two-phase method called cluster-first and route-second \cite{Fisher and Jaikumar 1981}. The LBBD decomposes the original problem into an assignment master problem in which the nodes are allocated to a number of required trips, and independent sequencing subproblems with the special structure of the traveling salesman problem considering AND-type PCs. At each iteration of the LBBD, the master problem is solved and provides a lower bound for the original problem, since it is a relaxation of the original problem. After solving the MP and specifying the nodes assignments to the trips, the subproblems (associated to the trips) are derived which provide an upper bound. Using the subproblems solution, the optimality cuts are developed to be added to the master problem. Such method converges to the optimum if and only if the master problem solution is improving the lower bound, and the added cuts are excluding the current master problem solution. This procedure goes on until an optimal solution is reached to the original problem or an early stopping criterion is found.

It should be noted that in the master problem and subproblems, the decisions are taken while satisfying the AND-type PCs among the trips and within them, respectively. More precisely, the nodes are allocated to the trips in the assignment master problem in such a way that the AND-type PCs between the couples of nodes are respected among the trips, while the PCs are met within each trip by the subproblem associated to that trip.

\subsection{Master problem} \label{master}
The master problem aims at making decisions on the nodes assignment to a number of required trips satisfying PCs among the trips. So, it involves the binary decision variables $x_{ir}, i \in N,  r \in \{1,...,p\}$ which takes value 1 if node $i$ is assigned to trip $r$, and the binary decision variables $u_r, r \in \{1,...,p\}$ which takes value $1$ if the vehicle travels trip $r$. The proposed master-problem is derived from the Integrated model, defined by constraints (\ref{1-integrated}-\ref{11-integrated}), and represented as follows:\\

\begin{equation}\label{obj-master}
 \min Z_{Master}
\end{equation}
subject to
\begin{equation} \label{1-master}
\begin{aligned}
  \sum_{r=1}^{p} x_{ir}=1  &&&&&&  \forall i \in N \setminus\{0\},
\end{aligned}
\end{equation}
\begin{equation} \label{2-master}
\begin{aligned}
  \sum_{i=1}^{N} d_i. x_{ir}\leq Q.u_r  &&&&&&  \forall r \in \{1,...,p\},
\end{aligned}
\end{equation}
\begin{equation} \label{3-master}
\begin{aligned}
u_{r} \leq u_{\acute{r}} &&&&&& \forall r,\acute{r} \in \{1,...,p\}_{r> \acute{r}},
\end{aligned}
\end{equation}
\begin{equation} \label{4-master}
\begin{aligned}
\sum_{k=1}^r x_{ek} \geq  x_{sr}  &&&&&&  \forall r \in \{1,...,p\}, e \in Pr_s,
\end{aligned}
\end{equation}
\begin{equation} \label{5-master}
\emph{Sub-Problem Relation},
\end{equation}
\begin{equation} \label{6-master}
\begin{aligned}
\emph{Optimality Cuts}   &&&&&&   \forall l \in Iterations,
\end{aligned}
\end{equation}
\begin{equation} \label{7-master}
\begin{aligned}
x_{ir}, u_r \in \{0,1\}   &&&&&&   &&&&&&  \forall r \in \{1,...,p\}, i \in N.
\end{aligned}
\end{equation}\\

As it can be seen, the MP involves constraints on the nodes allocation, the vehicle capacity, order of trips, AND-type PCs among the trips, as well as sub-problem relaxation and benders cuts.

Recently, a research proposed by Cire et al. \cite{Cire et al.2016} has showed that including the master problem with a relation of subproblem can significantly decrease the number of iterations as well as the computational time of LBBD as it provides a tighter lower bound for the original problem. Thus, we propose the following variables and the relaxed subproblem.

$y_{ijr} \in [0,1]$: continuous variable indicating the arc$(i,j)$ in trip $r$;

$\sigma_{r}$: integer variable for the lower bound on the travelling cost of trip $r$.\\

Then the model can stated as:

\begin{equation}\label{obj-relaxed-subproblem}
 \min \sum_{r=1}^{p} \sigma_r
\end{equation}
subject to
\begin{equation} \label{1-relaxed-subproblem}
\begin{aligned}
x_{jr}= \sum_{i=0, i\neq j}^{n} y_{ijr}  &&&&&&  \forall j \in N, r \in T,
\end{aligned}
\end{equation}
\begin{equation} \label{2-relaxed-subproblem}
\begin{aligned}
x_{ir}= \sum_{j=0, j\neq i}^{n} y_{ijr}  &&&&&&  \forall i \in N, r \in \{1,...,p\},
\end{aligned}
\end{equation}
\begin{equation}\label{3-relaxed-subproblem}
\begin{aligned}
\sigma_{r} \geq \mathop{\sum_{i=0}^{n} \sum_{j=0}^{n}}_{i\neq j} c_{ij} y_{ijr} &&&&&&  \forall r \in \{1,...,p\},
\end{aligned}
\end{equation}
\begin{equation}\label{4-relaxed-subproblem}
 Z_{Master} \geq \sum_{r=1}^{p} \sigma_r,
\end{equation}
\begin{equation} \label{5-relaxed-subproblem}
\begin{aligned}
x_{ir}\in \{0,1\}  &&&&&&  \forall i \in N, r \in \{1,...,p\},
\end{aligned}
\end{equation}
\begin{equation} \label{6-relaxed-subproblem}
\begin{aligned}
0 \leq y_{ijr} \leq 1  &&&&&&  \forall i,j \in N, r \in \{1,...,p\}.
\end{aligned}
\end{equation}\\
It should be noted that continuous variables $y_{ijr}$ are enforced to take binary values 0 or 1 due to the constraints (\ref{1-relaxed-subproblem}) and (\ref{2-relaxed-subproblem}) which are related to the binary assignment variables $x_{ir}$. Using constrains \eqref{3-relaxed-subproblem}, \eqref{4-relaxed-subproblem}, the lower bound for the cost of each trip and the value of master problem are computed, respectively.
Also, constrains (\ref{5-relaxed-subproblem}), (\ref{6-relaxed-subproblem}) define the binary and continuous variables, respectively.

It should be noted that the MP solutions may not be globally feasible in terms of the sequence of nodes in the trips, since no constraints are imposed to eliminate the possible sub-tours, and the PCs among the nodes within each trip.

\subsection{Sub-problem} \label{sub}
When the number of required trips and the nodes assignment to them are determined by the sub-optimal solution of the MP, each sub-problem associated to each trip can be seen as a traveling salesman problem with AND-type PCs. At each iteration, the sum of subproblems values (associated to the trips) provides an upper bound for the original problem.
In order to formulate the subproblem, $Alt_r^l$ and $nb_r^l$ are defined which, respectively, indicate the cluster (set) of nodes assigned to trip $r$ and its total number computed by the solution of master problem in iteration $l$. The two parameters are updated as the algorithm proceeds. For each iteration $l$ and each trip $r$ (with $u_r=1$), the sub-problem model and the variables are represented as follows:\\

$y_{ij}$: binary variable takes value $1$ if arc $(i,j)$ is linked, $0$ otherwise;

$po_i$: continuous variable indicating the position of node $i \in N$.\\

\begin{equation}\label{obj-subproblem}
  Z_{SP(r)}^l= \min \sum_{i=0}^{n} \sum_{j=0}^{n} c_{ij} y_{ij}
\end{equation}
subject to
\begin{equation} \label{1-subproblem}
\begin{aligned}
\sum_{i=0}^{n} y_{ij}=1  &&&&&&  \forall j \in Alt_r^l,
\end{aligned}
\end{equation}
\begin{equation} \label{2-subproblem}
\begin{aligned}
\sum_{j=0}^{n} y_{ij}=1  &&&&&&  \forall i \in Alt_r^l,
\end{aligned}
\end{equation}
\begin{equation} \label{3-subproblem}
\begin{aligned}
po_j - po_i +M (1-y_{ij}) \geq 1 &&&&&&  \forall i,j \in Alt_r^l,
\end{aligned}
\end{equation}
\begin{equation} \label{4-subproblem}
\begin{aligned}
po_j \geq po_i +1  &&&&&&  i,j \in Alt_r^l, i \in AND_j,
\end{aligned}
\end{equation}
\begin{equation} \label{5-subproblem}
\begin{aligned}
1 \leq po_i \leq nb_r^l &&&&&&  i \in Alt_r^l,
\end{aligned}
\end{equation}
\begin{equation}\label{6-subproblem}
\begin{aligned}
  po_{j} \geq 0 &&&&&& \forall j \in N,
\end{aligned}
\end{equation}
\begin{equation}\label{7-subproblem}
\begin{aligned}
  y_{ij} \in \{0,1\}  &&&&&& \forall i,j \in N.
\end{aligned}
\end{equation}\\

As usual objective function (\ref{obj-subproblem}) minimizes the total travelling costs for trip $r$. Equations (\ref{1-subproblem}) and (\ref{2-subproblem}) are the regular routing rules. Constraints (\ref{3-subproblem}), which relate the two types of variables, ensure the order of linked nodes in each trip. The precedence constraints are satisfied by imposing constrains (\ref{4-subproblem}). The position value is bounded by restriction (\ref{5-subproblem}). Constraints (\ref{3-subproblem}- \ref{5-subproblem}) enforce the continuous variable $Po_i, \forall i \in N$ takes the integer values between $0$ and $nb_r^l$. We empirically observed that considering $Po_i$ as continuous variable instead of integer one can reduce the computational time of the subproblem. Finally, constrains (\ref{6-subproblem}), (\ref{7-subproblem}) define the continuous and binary variables, respectively.

\subsection{Optimality Cuts} \label{benders cut}
The cuts play a vital role in the convergence of the LBBD algorithm. Unfortunately, there is no systematic procedure to derive such cuts as in classical Benders Decomposition. Therefore, tailored cuts must be developed according to the studied problem. There is no need to add any feasibility cut during the search since the subproblems are feasible for any master problem.

Let $Z_{SP(r)}^{l}$ and $\sigma_{r}$ denote, respectively, the optimal cost of the sequencing subproblem associated with trip $r$ and iteration $l$, and the cost of trip $r$. A simple variant of optimality cut is

\begin{equation} \label{cut0}
\sigma_{r} \geq Z_{SP(r)}^l -\sum_{i|x_{ir}^{l}=1} (1-x_{ir}),
\end{equation}
i.e., the MP gives the same solutions as before, $(1-x_{ir})$ for all $i$ becomes zero. It means a better result cannot be achieved as the objective function $\sigma_{r}$ being limited by $Z_{SP(r)}^l$. So, the master problem is enforced to change the nodes allocations till a new set, which has not been repeated before, is found.

Several researchers have used \eqref{cut0} and its equivalent version in different optimization problems, like \cite{Roshanaei et al.2017}, \cite{Fazel-Zarandi and Beck2012} and \cite{Barzanji et al.2020}. Unfortunately, the above-mentioned optimality cut \eqref{cut0} does not perform well in practice since the cut only affects a small number of solutions in the MP.
Recently, Faganello Fachini and Armentano \cite{Faganello Fachini and Armentano2020} have proposed an extended form of cut \eqref{cut0} in dealing with the heterogeneous fixed fleet vehicle routing problem with time windows. Their developed cut can be written as

\begin{equation} \label{cut1}
\sigma_{r} \geq Z_{SP(r)}^l -\sum_{i|x_{ir}^{l}=1} (1-x_{ir}) (2\max_{j\neq i|x_{jr}^{l}=1} \{c_{ij}\}).
\end{equation}

The cut indicates that if the current solution of the master problem does not include a node of a previously obtained trip, it reduces by twice the maximum travel cost over all pairs of nodes in this previously obtained route.

In this paper, we try to strengthen cut \eqref{cut1} to affect more solutions instead of excluding few ones. We propose the following optimality cuts for the master problem

\begin{eqnarray} \label{cut2}
\sigma_{r} \geq Z_{SP(r)}^l -\sum_{i|x_{ir}^{l}=1} (1-x_{ir}) (\mathop{max_{1} \{c_{ij}\}}_{j\neq i|x_{jr}^{l}=1}+\mathop{max_{2} \{c_{ij}\}}_{j\neq i|x_{jr}^{l}=1})+\\ \nonumber
 \sum_{i|x_{ir}^{l}=0} x_{ir} (\mathop{min_{1} \{c_{ij}\}}_{j\neq i|x_{jr}=1}+\mathop{min_{2} \{c_{ij}\}}_{j\neq i|x_{jr}=1}),
\end{eqnarray}
where $max_{1} \{c_{ij}\}$ and $max_{2} \{c_{ij}\}$ are the two highest, and $min_{1} \{c_{ij}\}$ and $min_{2} \{c_{ij}\}$ are the two lowest values of the traveling cost linked to node $i$. In comparison with cut \eqref{cut1}, a stronger lower bound for $\sigma_{r}$ can be obtained in \eqref{cut2}. This is due to considering not only the nodes are excluded from trip $r$ (by subtracting the two associated highest cost) but also the ones are newly included to that trip (by adding the two associated lowest cost). In addition, considering the two highest and the two lowest cost might drive tighter lower bound \eqref{cut2}, since

\begin{equation}\label{relation-cut-max}
\mathop{max_{1} \{c_{ij}\}}_{j\neq i|x_{jr}^{l}=1}+\mathop{max_{2} \{c_{ij}\}}_{j\neq i|x_{jr}^{l}=1} \leq 2\max_{j\neq i|x_{jr}^{l}=1} \{c_{ij}\},
\end{equation}

and

\begin{equation}\label{relation-cut-min}
\mathop{min_{1} \{c_{ij}\}}_{j\neq i|x_{jr}=1}+\mathop{min_{2} \{c_{ij}\}}_{j\neq i|x_{jr}=1} \geq 2\min_{j\neq i|x_{jr}=1} \{c_{ij}\}.
\end{equation}

To proof the cut validity, we somehow follow the same procedure proposed by \cite{Faganello-Fachini-and-Armentano2020}.

The Benders optimality cut communicates to the MP for two purposes: (1) eliminating the current sub-optimal MP solution (Theorem 1), (2) not removing the globally feasible solutions (Theorem 2).

\newtheorem{thm}{Theorem}
\begin{thm}
The optimality cut \eqref{cut2} remove the current sub-optimal solution.
\end{thm}

\begin{proof}
Let $\sigma_{r}^{l}, r \in \{1,...,p\}$ be the travelling cost of trip $r$ by the solution of master problem in iteration $l$. Solving the subproblems associated with all trips, three possible situations can be happened:
\begin{itemize}
\item $\sum_{r=1}^{p} \sigma_{r}^{l} > \sum_{r=1}^{p} Z_{SP(r)}^{l}$, which means that the LBBD terminates and the value of upper bound is optimal for the original problem;
\item $\sum_{r=1}^{p} \sigma_{r}^{l} = \sum_{r=1}^{p} Z_{SP(r)}^{l}$, which indicates that the LBBD terminates and the solution of master problem is optimal for the original problem;
\item $\sum_{r=1}^{p} \sigma_{r}^{l} < \sum_{r=1}^{p} Z_{SP(r)}^{l}$, i.e., the current solution of master problem is not globally feasible (sub-optimal solution). In this case , the optimality cuts derived from subproblems are added to the master problem. Lets suppose that the same solution (nodes allocation to trips) of master problem is obtained in the subsequent iteration, then for each trip $r$, $\sum_{i|x_{ir}^{l}=1} (1-x_{ir})=0$ and $\sum_{i|x_{ir}^{l}=0} x_{ir} =0$. From \eqref{cut2}, the optimality cut for each trip $r$ is
\begin{equation} \label{theorem1-cut1}
\sigma_{r} \geq Z_{SP(r)}^{l}.
\end{equation}
Then, we obtain
\begin{equation} \label{theorem1-cut2}
\sum_{r=1}^{p} \sigma_{r} \geq \sum_{r=1}^{p} Z_{SP(r)}^{l} > \sum_{r=1}^{p} \sigma_{r}^{l},
\end{equation}
which means that the current solution obtained in iteration $l$ is removed from the master problem space.
\end{itemize}
\end{proof}

\begin{thm}
The optimality cut \eqref{cut2} does not exclude any globally feasible solutions in future
iterations.
\end{thm}

\begin{proof}
We assume that there exists a globally feasible trip not satisfying the optimality cut and then showing a contradiction. Let $Z^{*}$ denote the optimal cost of trip $r$ associated with the cluster of customers $H^{*}$, which was assigned to that trip through a globally feasible solution for the original problem. Let $\bar{Z}$ be the optimal cost of the same trip associated with cluster $\bar{H}$, which was found in iteration $l$. We define three sets of nodes:

\begin{itemize}
\item $\phi_{1}=\{H^{*} \setminus \bar{H}\}$: nodes in $H^{*}$ not in $\bar{H}$;
\item $\phi_{2}=\{H^{*} \cap \bar{H}\}$: nodes in both $H^{*}$ and $\bar{H}$;
\item $\phi_{3}=\{\bar{H} \setminus H^{*}\}$: nodes in $\bar{H}$ not in $H^{*}$.
\end{itemize}

An optimality cut \eqref{cut2} is generated for trip $r$ associated with cluster $\bar{H}$ in iteration $l$. Let's assume that trip $r$ obtained from the cluster $H^{*}$ violate the optimality cut. Then, we have

\begin{equation} \label{theorem2-cut1}
Z^{*} < \bar{Z} -\sum_{i \in \bar{H}} (1-x_{ir}) (\mathop{max_{1} \{c_{ij}\}}_{j\neq i \in \bar{H}}+\mathop{max_{2} \{c_{ij}\}}_{j\neq i \in \bar{H}})+ \sum_{i \not \in \bar{H}} x_{ir} (\mathop{min_{1} \{c_{ij}\}}_{j\neq i \in H^{*}}+\mathop{min_{2} \{c_{ij}\}}_{j\neq i \in H^{*}}).
\end{equation}

For the nodes in set $\phi_{2}$, $\sum_{i \in \phi_{2}} (1-x_{ir})=0$ in the globally feasible Route$(H^{*})$. Since Route $(H^{*})$ dose not include the nodes in set $\phi_{3}$, therefore, $x_{ir}=0$ for all $i \in \phi_{3}$, and $\sum_{i \in \phi_{3}} (1-x_{ir})=|\phi_{3}|$. Moreover, set $\phi_{1}$ contains the nodes not included in $\bar{H}$. So, \eqref{theorem2-cut1} can be presented as

\begin{equation} \label{theorem2-cut2}
Z^{*} < \bar{Z} -\sum_{i \in \phi_{3}} (\mathop{max_{1} \{c_{ij}\}}_{j\neq i \in \bar{H}}+\mathop{max_{2} \{c_{ij}\}}_{j\neq i \in \bar{H}})+ \sum_{i \in \phi_{1}}(\mathop{min_{1} \{c_{ij}\}}_{j\neq i \in H^{*}}+\mathop{min_{2} \{c_{ij}\}}_{j\neq i \in H^{*}}).
\end{equation}

Note that if $\phi_{2}= \emptyset$, then $\phi_{3}= \bar{H}$ and $\phi_{1}=H^{*}$. We can define an upper bound for $\bar{Z}$ as

\begin{equation} \label{theorem2-cut3}
\bar{Z} \leq \sum_{i \in \phi_{3}}
(\mathop{max_{1} \{c_{ij}\}}_{j\neq i \in \phi_{3}}+\mathop{max_{2} \{c_{ij}\}}_{j\neq i \in \phi_{3}}),
\end{equation}

including the two highest costs of node $i \in \phi_{3}$ linked to other nodes in that set, and a lower bound for $Z^{*}$ as

\begin{equation} \label{theorem2-cut4}
Z^{*} \geq \sum_{i \in \phi_{1}} (\mathop{min_{1} \{c_{ij}\}}_{j\neq i \in \phi_{1}}+\mathop{min_{2} \{c_{ij}\}}_{j\neq i \in \phi_{1}}),
\end{equation}

involving the two lowest cost of node $i \in \phi_{1}$ linked to other nodes in that set. So, equation \eqref{theorem2-cut2} can be written as

\begin{equation} \label{theorem2-cut5}
Z^{*} < \bar{Z} - \sum_{i \in \phi_{3}}
(\mathop{max_{1} \{c_{ij}\}}_{j\neq i \in \phi_{3}}+\mathop{max_{2} \{c_{ij}\}}_{j\neq i \in \phi_{3}}) + \sum_{i \in \phi_{1}} (\mathop{min_{1} \{c_{ij}\}}_{j\neq i \in \phi_{1}}+\mathop{min_{2} \{c_{ij}\}}_{j\neq i \in \phi_{1}}).
\end{equation}

From \eqref{theorem2-cut3}, the value of $\bar{Z} - \sum_{i \in \phi_{3}}
(\mathop{max_{1} \{c_{ij}\}}_{j\neq i \in \phi_{3}}+\mathop{max_{2} \{c_{ij}\}}_{j\neq i \in \phi_{3}})$ is equal to or less than zero. Therefore, according to \eqref{theorem2-cut5}, $Z^{*}$ is less than a negative value plus $\sum_{i \in \phi_{1}} (\mathop{min_{1} \{c_{ij}\}}_{j\neq i \in \phi_{1}}+\mathop{min_{2} \{c_{ij}\}}_{j\neq i \in \phi_{1}})$ which leads to a contradiction with the lower bound of $Z^{*}$ defined in \eqref{theorem2-cut4}.\\

Otherwise $(\phi_{2} \neq \emptyset)$, the contradiction \eqref{theorem2-cut2} is more involved as we subsequently represent.

Given the globally feasible trip derived from a cluster of nodes, it is possible to reach a reduced or an extended route by removing or adding the nodes to that cluster. Let's consider the trip associated with cluster $H^{*}$ with cost $Z^{*}$. By removing the nodes in $\phi_{1}$ from $H^{*}$, Trip$(\phi_{2})$ and the respective cost $Z_{\phi_{2}}$ is obtained. Removing the nodes $i \in \phi_{1}$ from Trip$(H^{*})$ results in a cost smaller than or equal to $Z^{*}$, since travel costs satisfy the triangular inequality, therefore,

\begin{equation} \label{theorem2-cut6}
Z_{\phi_{2}}=Z_{H^{*} \setminus \phi_{1}} \leq Z^{*}-\sum_{i \in \phi_{1}}(\mathop{min_{1} \{c_{ij}\}}_{j\neq i \in H^{*}}+\mathop{min_{2} \{c_{ij}\}}_{j\neq i \in H^{*}}).
\end{equation}

The above equation indicates that removing nodes in $\phi_{1}$, consequently, subtracting the two lowest cost of node $i$ linked to node $j \in H^{*}$ from the value of $Z^{*}$ defines an upper bound for the cost $Z_{\phi_{2}}$.

Accordingly, \eqref{theorem2-cut6}, using the contradiction \eqref{theorem2-cut2}, becomes:

\begin{eqnarray} \label{theorem2-cut7}
Z_{\phi_{2}} &\leq& Z^{*}-\sum_{i \in \phi_{1}}(\mathop{min_{1} \{c_{ij}\}}_{j\neq i \in H^{*}}+\mathop{min_{2} \{c_{ij}\}}_{j\neq i \in H^{*}}) \leq Z^{*} \nonumber\\
&\leq& \bar{Z} -\sum_{i \in \phi_{3}} (\mathop{max_{1} \{c_{ij}\}}_{j\neq i \in \bar{H}}+\mathop{max_{2} \{c_{ij}\}}_{j\neq i \in \bar{H}}) \nonumber\\
&+& \sum_{i \in \phi_{1}}(\mathop{min_{1} \{c_{ij}\}}_{j\neq i \in H^{*}}+\mathop{min_{2} \{c_{ij}\}}_{j\neq i \in H^{*}}).
\end{eqnarray}

Now let us extend Trip$(\phi_{2})$ to Trip$(\phi_{2}\cup \phi_{3})$ with cost $Z_{\phi_{2}\cup \phi_{3}}$ by inserting the nodes belongs to $\phi_{3}$ into Trip$(\phi_{2})$. Consider Trip$(\phi_{2})$ represented by the sequence of nodes $(0,...,g,o,...,z,0)$. Let's assume that the node $i \in \phi_{3}$ is inserted into Trip$(\phi_{2})$ such that $\mathop{max_{1} \{c_{ij}\}}_{j\neq i \in \phi_{2}}+\mathop{max_{2} \{c_{ij}\}}_{j\neq i \in \phi_{2}}$ is the sum of the two highest cost of $c_{i0},...,c_{ig},c_{io},...,c_{iz}$. The insertion cost of node $i \in \phi_{3}$ between nodes $g$ and $o$ is given by $c_{gi}+c_{io}-c_{go}$.
Moreover, observe that $c_{gi}-c_{go}\leq c_{oi}$, since the triangular inequality holds.
Therefore, the insertion cost of node $i \in \phi_{3}$ between nodes $g$ and $o$ yields

\begin{eqnarray} \label{theorem2-cut77}
c_{gi}+c_{io}-c_{go}\leq c_{gi}+c_{io} \leq
\mathop{max_{1} \{c_{ij}\}}_{j\neq i \in \phi_{2}}+\mathop{max_{2} \{c_{ij}\}}_{j\neq i \in \phi_{2}}.
\end{eqnarray}

Proceeding this way for all nodes in $\phi_{3}$, we obtain Trip$(\phi_{2}\cup \phi_{3})$ such that

\begin{equation} \label{theorem2-cut8}
Z_{(\phi_{2}\cup \phi_{3})} \leq Z^{*}
+ \sum_{i \in \phi_{3}} (\mathop{max_{1} \{c_{ij}\}}_{j\neq i \in \phi_{2}}+\mathop{max_{2} \{c_{ij}\}}_{j\neq i \in \phi_{2}}) -\sum_{i \in \phi_{1}}(\mathop{min_{1} \{c_{ij}\}}_{j\neq i \in H^{*}}+\mathop{min_{2} \{c_{ij}\}}_{j\neq i \in H^{*}}).
\end{equation}

Since $\phi_{2} \subset (\phi_{2}\cup \phi_{3})=\bar{H}$, for all $i \in \phi_{3}$, one has

\begin{eqnarray} \label{theorem2-cut88}
\mathop{max_{1} \{c_{ij}\}}_{j\neq i \in \phi_{2}}+\mathop{max_{2} \{c_{ij}\}}_{j\neq i \in \phi_{2}}) \leq \mathop{max_{1} \{c_{ij}\}}_{j\neq i \in \bar{H}}+\mathop{max_{2} \{c_{ij}\}}_{j\neq i \in \bar{H}}).
\end{eqnarray}

Hence, \eqref{theorem2-cut8} can be written as

\begin{equation} \label{theorem2-cut9}
Z_{(\phi_{2}\cup \phi_{3})} \leq Z^{*}-\sum_{i \in \phi_{1}}(\mathop{min_{1} \{c_{ij}\}}_{j\neq i \in H^{*}}+\mathop{min_{2} \{c_{ij}\}}_{j\neq i \in H^{*}})
+ \sum_{i \in \phi_{3}} (\mathop{max_{1} \{c_{ij}\}}_{j\neq i \in \bar{H}}+\mathop{max_{2} \{c_{ij}\}}_{j\neq i \in \bar{H}}).
\end{equation}

Regarding \eqref{theorem2-cut9}, one gets

\begin{equation} \label{theorem2-cut10}
Z^{*} \geq Z_{(\phi_{2}\cup \phi_{3})} +\sum_{i \in \phi_{1}}(\mathop{min_{1} \{c_{ij}\}}_{j\neq i \in H^{*}}+\mathop{min_{2} \{c_{ij}\}}_{j\neq i \in H^{*}})
-\sum_{i \in \phi_{3}} (\mathop{max_{1} \{c_{ij}\}}_{j\neq i \in \bar{H}}+\mathop{max_{2} \{c_{ij}\}}_{j\neq i \in \bar{H}}).
\end{equation}

As Route$(\bar{H})$ is optimal for the cluster $\bar{H}=\phi_{2}\cup \phi_{3}$ with cost $\bar{Z}$, then

\begin{equation} \label{theorem2-cut11}
Z_{(\phi_{2}\cup \phi_{3})} \geq \bar{Z}.
\end{equation}

From \eqref{theorem2-cut10} and \eqref{theorem2-cut11} we obtain

\begin{equation} \label{theorem2-cut12}
Z^{*} \geq \bar{Z} +\sum_{i \in \phi_{1}}(\mathop{min_{1} \{c_{ij}\}}_{j\neq i \in H^{*}}+\mathop{min_{2} \{c_{ij}\}}_{j\neq i \in H^{*}})
-\sum_{i \in \phi_{3}} (\mathop{max_{1} \{c_{ij}\}}_{j\neq i \in \bar{H}}+\mathop{max_{2} \{c_{ij}\}}_{j\neq i \in \bar{H}}).
\end{equation}

Equation \eqref{theorem2-cut12} contradicts \eqref{theorem2-cut7} and, consequently, the optimality cut \eqref{cut2} does not remove globally feasible solutions.
\end{proof}

\section{Computational experiments} \label{Computational experiments}
This section evaluates the performance of the proposed LBBD algorithm against the two-index MIP model, the best formulation among the three proposed ones as shown in section \ref{models comparison}. We test three different versions of the LBBD algorithm. The first one denoted by LBBD1 incorporates the optimality cut \eqref{cut1} proposed recently by Faganello Fachini and Armentano \cite{Faganello Fachini and Armentano2020}. In LBBD2, our developed cut, represented as \eqref{cut2}, is applied. The two algorithms LBBD1 and LBBD2 are the exact approaches in which both the MP and the SP are optimally solved (optimality gap of zero).
LBBD3 is a heuristic version of LBBD2 in which the SP is optimally solved, yet the MP is
solved for at most 5 seconds or 5\% optimality gap (whichever comes first). This change to the master problem results in it no longer being a true lower bound. Such a stopping criterion was first proposed in \cite{Tran2012} under which they speeded up the MP by allowing the solver to stop once a solution is found within some predetermined gap from the best lower bound obtained. However, they indicated that the quality of the solution found is within the chosen optimality gap. Moreover, Barzanji et al. \cite{Barzanji et al.2020} has recently improved their proposed LBBD by defining a gap of 5\% for solving the SP in dealing with the integrated process planning and scheduling problem. As suggested in \cite{Tran2012} and \cite{Barzanji et al.2020} and, also according to preliminary experiments, a gap of 5\% is chosen, in this research, which results in a good trade off between computational time and quality.

The model and the algorithms are coded in GAMS 24.1 on corei5 pc with 2.50 GHz CPU and 4 GB of RAM, and the time limit is set to maximum 5400 seconds.
In section \ref{Instance generation}, we describe the instances generated as a testbed for our assessment. Our computational experiment results are described and commented on in Section
\ref{comparison}.

\subsection{Instance generation} \label{Instance generation}
Since we could not find a the standard benchmark for the proposed problem, our test-bed is generated by extending the same method described in \cite{Augerat et al.1995} to deal with the CVRP. To build our instances, we assume that there is a single vehicle with the possibility of multiple trips, then, we take the proposed number of vehicles as the maximum number of trips in our problem. There are 27 instances having the nodes ranging from 32 to 80 randomly generated from points with real coordinates using a uniform distribution in the range $[0,100]$. Rounded Euclidean distances are taken as travel cost between each pair of nodes. The vehicle capacity $Q$ is set to 100 for all instances. Demand $d_i, \forall i \in N \setminus\{0\}$ is picked from an uniform distribution $U(1,30)$, however $n/10$ of those demands are multiplied by 3.

For each instance, the Precedence Matrix (PM) is generated using three different values of scale $\tau =\{0,0.4,0.8\}$ to evaluate the PCs' impact on the complexity of instances. The procedure is the same described in section \ref{models comparison}. As a result, the total number of instances is $27 \times 3= 81$, as each instance is solved with three proposed scales of precedence constraints $\tau \in \{0,0.4,0.8\}$. All the instances are solved using the Two-index MIP model and the three algorithm LBBD1, LBBD2, and LBBD3.

\subsection{Comparison of Two-index MIP model and three LBBD algorithms} \label{comparison}
In this section, we summarize and discuss our experimental results to evaluate the performance of the proposed algorithms in comparison with the Two-index MIP model on the sets of instances described in Section \ref{Instance generation}.

Tables \ref{table-comparison1}-\ref{table-comparison3} represents the results of experiments associated to the three proposed scales of PCs. In all the tables, column 1 displays the number of nodes, while column 2 indicates the proposed maximum number of trips for each instance. The CPU time (in seconds) is shown in columns 3, 5, 7 and 9, respectively for the Two-index model, LBBD1, LBBD2, and LBBD3. A dash in these columns indicates the proposed time limit under which the model was unable to find optimal solutions.
For each instance, the optimality gap is calculated with the following formula.

\begin{equation}\label{optimality gap}
\textit{optimality gap} = \frac{\textit{Upper bound} - \textit{Lower bound}}{\textit{Upper bound}} \times 100.
\end{equation}

Values in columns 4, 6, 8, and 10 represents the optimality gap (denoted by $\varDelta(\%)$), respectively for the Two-index model, LBBD1, LBBD2, and LBBD3.\\

\begin{table}
\caption{Computational results of the LBBD1, LBBD2, LBBD3 and the Two-index model considering $\tau =0$.}
\label{table-comparison1}
\scalebox{0.75}{
\begin{tabular}{rrrrrrrrrr}
\hline\noalign{\smallskip}
 \multicolumn {2}{c}{Instances}& \multicolumn {2}{c}{Two-index model}&\multicolumn {2}{c}{LBBD1} & \multicolumn {2}{c}{LBBD2} & \multicolumn {2}{c}{LBBD3}\\
  \noalign{\smallskip} \cline{1-2} \cline{3-4} \cline{5-6} \cline{7-8} \cline{9-10}\noalign{\smallskip}
   $n$ & $p$ & CPU time $(s)$ & $\varDelta(\%)$ & CPU time $(s)$ & $\varDelta(\%)$ & CPU time $(s)$ & $\varDelta(\%)$ &CPU time $(s)$ & $\varDelta(\%)$\\
\noalign{\smallskip}\hline\noalign{\smallskip}
 32 &5 &1804& 0  &835 & 0 &746  & 0 & 470 & 0 \\
 33 &5 &2117& 0  &794 & 0 &762  & 0 & 496 & 0 \\
 33 &6 & -  & 13 &863 & 0 &859  & 0 & 509 & 0 \\
 34 &5 & -  & 11 &807 & 0 &693  & 0 & 431 & 0 \\
 36 &5 &3943& 0  &783 & 0 &754  & 0 & 457 & 0 \\
 37 &5 & -  & 10 &915 & 0 &833  & 0 & 524 & 0 \\
 37 &6 & -  & 11 &986 & 0 &875  & 0 & 648 & 0 \\
 38 &5 & -  & 12 &887 & 0 &835  & 0 & 502 & 0 \\
 39 &5 &4820& 0  &936 & 0 &917  & 0 & 530 & 0 \\
 39 &6 & -  & 12 &973 & 0 &940  & 0 & 597 & 0 \\
 44 &6 & -  & 14 &1035& 0 &987  & 0 & 661 & 0 \\
 45 &6 & -  & 13 &1090& 0 &1006 & 0 & 735 & 0 \\
 45 &7 & -  & 14 &1271& 0 &998  & 0 & 672 & 0 \\
 46 &7 & -  & 16 &1183& 0 &1055 & 0 & 704 & 0 \\
 48 &7 & -  & 22 &1256& 0 &1107 & 0 & 682 & 0 \\
 53 &7 & -  & 24 &1319& 0 &1143 & 0 & 704 & 0 \\
 54 &7 & -  & 23 &1338& 0 &1117 & 0 & 756 & 0 \\
 55 &9 & -  & 28 &1952& 0 &1736 & 0 & 962 & 0 \\
 60 &9 & -  & 31 &1844& 0 &1771 & 0 & 1035& 0 \\
 61 &9 & -  & 29 &1996& 0 &1550 & 0 & 964 & 0 \\
 62 &8 & -  & 27 &1638& 0 &1300 & 0 & 711 & 0 \\
 63 &9 & -  & 34 & -  & 12&2644 & 0 & 805 & 0 \\
 63 &10& -  & 38 & -  & 8 &2930 & 0 & 946 & 0 \\
 64 &9 & -  & 35 &1950& 0 &1733 & 0 & 1061& 0 \\
 65 &9 & -  & 36 & -  & 10&1852 & 0 & 1143& 0 \\
 69 &9 & -  & 39 &1826& 0 &1677 & 0 & 972 & 0 \\
 80 &10& -  & 37 & -  & 14& -   & 9 & 2605& 0\\
\noalign{\smallskip}\hline
\end{tabular}}
\end{table}

\begin{table}
\caption{Computational results of the LBBD1, LBBD2, LBBD3 and the Two-index model considering $\tau =0.4$.}
\label{table-comparison2}
\scalebox{0.75}{
\begin{tabular}{rrrrrrrrrr}
\hline\noalign{\smallskip}
 \multicolumn {2}{c}{Instances}& \multicolumn {2}{c}{Two-index model}&\multicolumn {2}{c}{LBBD1} & \multicolumn {2}{c}{LBBD2} & \multicolumn {2}{c}{LBBD3}\\
  \noalign{\smallskip} \cline{1-2} \cline{3-4} \cline{5-6} \cline{7-8} \cline{9-10}\noalign{\smallskip}
   $n$ & $p$ & CPU time $(s)$ & $\varDelta(\%)$ & CPU time $(s)$ & $\varDelta(\%)$ & CPU time $(s)$ & $\varDelta(\%)$ &CPU time $(s)$ & $\varDelta(\%)$\\
\noalign{\smallskip}\hline\noalign{\smallskip}
 32 &5 &2289& 0  & 972 & 0 & 815 & 0  & 397  & 0 \\
 33 &5 &1837& 0  & 858 & 0 & 739 & 0  & 466  & 0 \\
 33 &6 &3035& 0  & 793 & 0 & 645 & 0  & 503  & 0 \\
 34 &5 &-   & 14 & 927 & 0 & 732 & 0  & 628  & 0 \\
 36 &5 &4273& 0  & 740 & 0 & 712 & 0  & 413  & 0 \\
 37 &5 &-   & 9  & 856 & 0 & 784 & 0  & 475  & 0 \\
 37 &6 &-   & 10 & 891 & 0 & 839 & 0  & 536  & 0 \\
 38 &5 &3792& 0  & 863 & 0 & 807 & 0  & 570  & 0 \\
 39 &5 &5008& 0  & 957 & 0 & 840 & 0  & 637  & 0 \\
 39 &6 & -  & 14 & 1016& 0 & 917 & 0  & 716  & 0 \\
 44 &6 & -  & 16 & 993 & 0 & 846 & 0  & 741  & 0 \\
 45 &6 & -  & 13 & 1072& 0 & 968 & 0  & 780  & 0 \\
 45 &7 & -  & 17 & 1335& 0 &1152 & 0  & 833  & 0 \\
 46 &7 & -  & 20 & 1494& 0 &1305 & 0  & 1025 & 0 \\
 48 &7 & -  & 18 & 1039& 0 & 876 & 0  & 668  & 0 \\
 53 &7 & -  & 22 & 1266& 0 &1115 & 0  & 916  & 0 \\
 54 &7 & -  & 25 & 1392& 0 &1248 & 0  & 1037 & 0 \\
 55 &9 & -  & 29 & 1773& 0 &1566 & 0  & 1243 & 0 \\
 60 &9 & -  & 33 & 1827& 0 &1681 & 0  & 1469 & 0 \\
 61 &9 & -  & 30 & 1875& 0 &1736 & 0  & 1370 & 0 \\
 62 &8 & -  & 28 & 1745& 0 &1624 & 0  & 1288 & 0 \\
 63 &9 & -  & 31 & -   &15 & -   & 9  & 2726 & 0 \\
 63 &10& -  & 35 & -   & 6 &2072 & 0  & 1655 & 0 \\
 64 &9 & -  & 29 & 1713& 0 &1405 & 0  & 1168 & 0 \\
 65 &9 & -  & 26 & -   &16 & -   & 7  & 2339 & 0 \\
 69 &9 & -  & 35 & 1645& 0 &1490 & 0  & 1173 & 0 \\
 80 &10& -  & 39 & -   &17 & -   & 11 & 2902 & 0 \\
\noalign{\smallskip}\hline
\end{tabular}}
\end{table}

\begin{table}
	\caption{Computational results of the LBBD1, LBBD2, LBBD3 and the Two-index model considering $\tau =0.8$.}
	 \label{table-comparison3}
\scalebox{0.75}{
\begin{tabular}{rrrrrrrrrr}
\hline\noalign{\smallskip}
 \multicolumn {2}{c}{Instances}& \multicolumn {2}{c}{Two-index model}&\multicolumn {2}{c}{LBBD1} & \multicolumn {2}{c}{LBBD2} & \multicolumn {2}{c}{LBBD3}\\
  \noalign{\smallskip} \cline{1-2} \cline{3-4} \cline{5-6} \cline{7-8} \cline{9-10}\noalign{\smallskip}
   $n$ & $p$ & CPU time $(s)$ & $\varDelta(\%)$ & CPU time $(s)$ & $\varDelta(\%)$ & CPU time $(s)$ & $\varDelta(\%)$ &CPU time $(s)$ & $\varDelta(\%)$\\
\noalign{\smallskip}\hline\noalign{\smallskip}
 32 &5 &1998& 0   & 950  & 0 & 923 & 0  & 551 & 0 \\
 33 &5 &2218& 0   & 981  & 0 & 576 & 0  & 352 & 0 \\
 33 &6 &3249& 0   & 883  & 0 & 806 & 0  & 714 & 0 \\
 34 &5 & -  & 16  & 956  & 0 & 851 & 0  & 693 & 0  \\
 36 &5 &4540& 0   & 722  & 0 & 647 & 0  & 478 & 0 \\
 37 &5 & -  & 12  & 826  & 0 & 713 & 0  & 639 & 0 \\
 37 &6 & -  & 8   & 875  & 0 & 822 & 0  & 540 & 0 \\
 38 &5 & -  & 14  & 836  & 0 & 783 & 0  & 633 & 0 \\
 39 &5 &4712& 0   & 972  & 0 & 852 & 0  & 689 & 0 \\
 39 &6 & -  & 17  & 990  & 0 & 964 & 0  & 616 & 0 \\
 44 &6 & -  & 15  & 1072 & 0 & 907 & 0  & 762 & 0 \\
 45 &6 & -  & 16  & 1039 & 0 & 932 & 0  & 750 & 0 \\
 45 &7 & -  & 19  & 1378 & 0 & 1105& 0  & 961 & 0 \\
 46 &7 & -  & 20  & 1422 & 0 & 1213& 0  & 1037& 0 \\
 48 &7 & -  & 23  & 1156 & 0 & 1044& 0  & 972 & 0 \\
 53 &7 & -  & 26  & 1437 & 0 & 1267& 0  & 958 & 0 \\
 54 &7 & -  & 23  & 1335 & 0 & 1105& 0  & 966 & 0 \\
 55 &9 & -  & 27  & 1822 & 0 & 1579& 0  & 1203& 0 \\
 60 &9 & -  & 38  & 1905 & 0 & 1601& 0  & 1355& 0 \\
 61 &9 & -  & 29  & 2016 & 0 & 1871& 0  & 1290& 0 \\
 62 &8 & -  & 28  & 1590 & 0 & 1358& 0  & 1083& 0 \\
 63 &9 & -  & 33  &   -  & 11&  -  & 8  & 2295& 0 \\
 63 &10& -  & 30  & 3057 & 0 & 2472& 0  & 1832& 0 \\
 64 &9 & -  & 34  & 1582 & 0 & 1338& 0  & 1172& 0 \\
 65 &9 & -  & 30  &   -  & 15&  -  & 8  & 2489& 0 \\
 69 &9 & -  & 31  & 1937 & 0 & 1824& 0  & 1362& 0 \\
 80 &10& -  & 36  &   -  & 13&  -  & 9  & 2630& 0 \\
\noalign{\smallskip}\hline
\end{tabular}}
\end{table}

The first thing that should be noticed is that the results of the three tables highlight way better performance of the LBBD3 comparing the other approaches as it is able to heuristically solve all the instances in less computational time. Comparing the two algorithms LBBD1 and LBBD2 indicates the higher efficiency of the LBBD2 capable of solving almost all instances in less time and even some instances which can not be optimally dealt with the LBBD1. This is because of the cut \eqref{cut2} used in LBBD2 which can provide tighter lower bound for the master problem than optimality cut \eqref{cut1} applied in the LBBD1. Not surprisingly, the Two-index model can only solve smaller instances with more CPU time comparing to the LBBD algorithms for all the proposed PCs scales.

Moreover, it can be observed an increasing trend in the problem complexity as the size of instances grows in all the three tables. More precisely, raising the number of nodes and the maximum number of trips may lead to more CPU time and the optimality gap, except for a few cases in which the size of instances is close together. As expected, the parameter of maximum trips number can affect the problem complexity more than the number of nodes in most cases for all the proposed solution approaches.

Considering different values of PC scale, it can be observed that the complexity of instances are not necessarily affected by increasing or decreasing the PC scale value. In some instances, larger value of PC scale leads to low computational complexity, consequently less CPU time, and in some other ones, the opposite is happened. This behaviour is the same for all the proposed approaches. For instance, let's consider the instance with 48 nodes and at most 7 trips solved using algorithm LBBD2. As it is shown, the CPU time is equal to 1107 when no PC is considered, while the time decreases once $\tau = 0.4$ and increases when $\tau = 0.8$. It seems, in this case, the instance complexity reduces considering precedence relations in comparison with the non-PC problem. As mentioned in section \ref{models comparison}, not the number of PCs but the structure of them on the graph may affect differently the problem complexity. This topic has also been addressed in \cite{Prot and Bellenguez-Morineau 2018}. According to their research, many problems can be NP-hard when considering general precedence constraints, while they become polynomially solvable for particular structures of precedence constraints.

\section{Conclusions} \label{conclusion}
In this paper, the multi-trip single vehicle routing problem is studied in which the nodes associated to customers/target locations are related to each other through AND-type precedence constraints. Our interest originates from problems, such as package delivery or picker routing problems, where some nodes have priorities to be visited after a set of other ones within and among the routes, due to their emergency, importance or any physical features. Despite of the AND-type PCs applications in real-life routing problems, non of the studies in the literature of logistics, even picker routing problems, focuses on these relations.

First, we develop and experimentally compare three mathematical formulations to address the problem. The computational results validates the significant superiority of the developed Two-index model in terms of both computational time and problem size in comparison with the two other ones. Then, the problem is handled by proposing a solution approach based on the logic based benders decomposition algorithm. The developed approach decomposes the original problem into an assignment master problem and sequencing subproblems. Moreover, a new optimality cut is provided, and its validity is proven. The performance of the optimality cut is experimentally investigated by comparing that with a recently proposed cut in the literature. Additionally, we presented a relaxed version of LBBD by defining a limit for optimality gap and CPU time in deriving master-problem solutions. In such a way the efficiency of algorithm improves as it allows the algorithm to find feasible solution of original problem in less CPU time and even for larger instances. The performance of proposed LBBD algorithms are evaluated and compared together through extensive computational experiments. The results show that the two exact LBBDs are able to solve most instances, while the relaxed version of LBBD can heuristically solve all the generated instances in shorter computational time.

Since this research is the first attempt in proposing AND-type PCs in routing area, various future topics can be explored, including different variants of vehicle routing problems and even real-life applications. Designing and developing other exact solution approaches, well-known heuristic and meta-heuristic algorithms along with enhancements in both parts of master and subproblem of the proposed LBBD algorithm are highly recommended as future research studies.


\begin{thebibliography}{}
\bibitem{Agnetis et al. 2019}
Agnetis, A., Rossi, F., Smriglio, S., Some Results on Shop Scheduling with S-Precedence Constraints among Job Tasks, Algorithms, 12, 1--12 (2019).
\bibitem{Angel-Bello et al. 2013}
Angel-Bello, F., Martinez-Salazar, I., Alvarez, A., Minimizing waiting times in a route design problem with multiple use of a single vehicle, Journal of Electronic Notes in Discrete Mathematics, 41, 269–-276 (2013).
\bibitem{Augerat et al.1995}
Augerat P., Belenguer J., Benavent E., Corberan, A., Naddef D., Rinaldi G., Computational results with a branch and cut code for the capacitated vehicle routing problem, Tech. Rep. 949-M, Universit e Joseph Fourier, Grenoble, France, (1995).
\bibitem{Azi et al. 2007}
Azi, N., Gendreau, M., Potvin, J-Y., Schneide, M., An exact algorithm for a single-vehicle routing problem with time windows and multiple routes, European Journal of Operational Research, 178, 755–-766 (2007).
\bibitem{Barzanji et al.2020}
Barzanji R., Naderi B., A.Begen M., Decomposition algorithms for the integrated process planning and scheduling problem, Omega, 93, 1–-13 (2020).
\bibitem{Bai et al. 2019}
Bai X., Cao M., S-Ge S., Efficient Routing for Precedence-Constrained Package Delivery for Heterogeneous Vehicles, Journal of IEEE Transaction on Automation, 1--13 (2019).
\bibitem{Bockenhauer et al. 2013}
Bockenhauer H.J., Momke T., Steinova M., Improved approximations for TSP with simple precedence constraints, Journal of Discrete Algorithms, 21, 32–-40 (2013).
\bibitem{Bredstrom and Ronnqvist 2008}
Bredstrom D., Ronnqvist M., Combined vehicle routing and scheduling with temporal precedence and synchronization constraints, European Journal of Opertional Research, 191, 19–-31 (2008).
\bibitem{Cattaruzz2016}
Cattaruzz, D., Absi, N., Feillet, D., Vehicle routing problems with multiple trips, Journal of Operations Research, 14, 223--259 (2016).
\bibitem{Cattaruzz2018}
Cattaruzz, D., Absi N., Feillet D., Vehicle routing problems with multiple trips, Journal of Annals of Operations Research, 271, 127–-159 (2018).
\bibitem{Chabot et al. 2017}
Chabot T., Lahyani R., Coelho L.C., Renaud J., Order picking problems under weight, fragility and category constraints, International Journal of Production Research, 55, 6361--6379 (2017).
\bibitem{Cheng2009}
Cheng, C.B., Wang, K.P., Solving a vehicle routing problem with time windows by a decomposition technique and a genetic algorithm, Expert Systems with Applications, 36, 7758--7763 (2009).
\bibitem{Casazza2018}
Casazza, M., Ceselli, A., Wolfler-Calvo, R., A branch and price approach for the Split Pickup and Split Delivery VRP, Electronic Notes in Discrete Mathematics, 69, 189--196 (2018).
\bibitem{Cire et al.2016}
Cire, A.A., Coban, E., Hooker, J.N., Logic-based Benders decomposition for planning and scheduling: A computational analysis, The Knowledge Engineering Review, 31, 440–-451 (2016).
\bibitem{Cordeau and Laporte 2007}
Cordeau J.F., Laporte G., The dial-a-ride problem: Models and algorithms, Journal of Annaul operational Research, 153, 29–-46 (2007).
\bibitem{De-Camargo2013}
De-Camargo, R.S., de-Miranda, G., Lokketangen, A., A new formulation and an exact approach for the many-to-many hub location-routing problem, Applied Mathematical Modelling, 37, 7465--7480 (2013).
\bibitem{Dekker et al. 2004}
Dekker R., deKoster M.B.M., Roodbergen K.J., Van Kalleveen, H., Improving order-picking response time at ankor's warehouse, Interfaces, 34, 303--313 (2004).
\bibitem{Derriesel and Monch 2011}
Derriesel R., Monch L., Variable neighbourhood search approaches for scheduling jobs on parallel machines with sequence-dependent setup times, precedence constraints, and ready times, Computers and Industrial Engineering, 61, 336--345 (2011).
\bibitem{Dong2017}
Dong, Y., Maravelias, C.T., Pinto, J.M., Sundaramoorthy, A., Solution methods for vehicle-based inventory routing problems, Computers and Chemical Engineering, 101, 259--278 (2017).
\bibitem{Fagerholt and Christiansen 2000}
Fagerholt K., Christiansen M., A travelling salesman problem with allocation time window and precedence constraints an application to ship scheduling, Journal of International Transactions in Operational Research,7 ,231–-244 (2000).
\bibitem{Fazel-Zarandi and Beck2012}
Fazel-Zarandi M., Beck C., Using Logic-Based Benders Decomposition to Solve the Capacity- and Distance-Constrained Plant Location Problem, INFORMS Journal on Computing, 24, 387–-398 (2012).
\bibitem{Faganello Fachini and Armentano2020}
Faganello Fachini R., Armentano V.A., Logic-based Benders decomposition for the heterogeneous fixed fleet vehicle routing problem with time windows, Computers and Industrial Engineering, 148, 1–-18 (2020).
\bibitem{Fisher and Jaikumar 1981}
Fisher R., Jaikumar V.A., A generalized assignment heuristic for vehicle routing, Networks, 11, 109–-124 (1981).
\bibitem{Gedik et al.2016}
Gedik, R., Rainwater, C., Nachtmann, H., Pohl, E.A., Analysis of a parallel machine scheduling problem with sequence dependent setup times and job availability intervals, European Journal of Operational Research, 251, 640--650 (2016).
\bibitem{Gillies and Liu 1995}
Gillies D.W., Liu W.S., Scheduling tasks with AND/OR precedence constraints, SIAM Journal on Computing, 24, 797--810 (1995).
\bibitem{Gordon2008}
Gordon, V. S., Potts, C.N., Strusevich, V.A., Whitehead, J.D., Single machine scheduling models with deterioration and learning: Handling precedence constraints via priority generation, Scheduling, 11, 357--370 (2008).
\bibitem{Goldwasser and Motwani 1999}
Goldwasser M.H., Motwani R., Complexity measures for assembly sequences, International Journal of Computational Geometry and Applications, 9, 371--418 (1999).
\bibitem{Hintsch2018}
Hintsch, T., Irnich, S., Large multiple neighborhood search for the clustered vehicle-routing problem, European Journal of Operational Research, 270, 118--131 (2018).
\bibitem{Ho et al. 2018}
Ho, S., Szeto, W.Y., Luo, Y.H., Leung J.M.Y., Petering M., Tou J.M.Y., A survey of dial-a-ride problems: Literature review and recent developments, Transportation Research Part B, 111, 395--421 (2018).
\bibitem{Kallehauge2005}
Kallehauge, B., Larsen, J., Madsen, O.B.G., Solomon M.M., Vehicle Routing Problem with Time Windows, Column Generation. Springer, Boston, MA, 67--98 (2005).
\bibitem{Lee et al. 2012}
Lee S., Moon I., Bae H., Kim J., Flexible job-shop scheduling problems with AND/OR precedence constraints, International Journal of Production Research, 50, 1979--2001 (2012).
\bibitem{Martinez2019}
Martinez, K.P., Adulyasak, Y., Jans, R., Morabito, R., Toso, E.A.V., An exact optimization approach for an integrated process configuration, lot-sizing, and scheduling problem, Computers and Operations Research, 103, 310--323 (2019).
\bibitem{Martinez-Salazar et al. 2014}
Martinez-Salazar, I., Angel-Bello, F., Alvarez, A., Schneide M., A customer-centric routing problem with multiple trips of a single vehicle, Journal of the Operational Research Society, 66, 1312--1323 (2014).
\bibitem{Matusiak et al. 2014}
Matusiak M., De-Koster R., Saarinen J., A fast simulated annealing method for batching precedence-constrained customer orders in a warehouse, European Journal of Operational Research, 236, 968--977 (2014).
\bibitem{Mingozzi et al.1997}
Mingozzi A., Bianco L., Ricciardelli A., Dynamic programming strategies for the TSP with time windows and precedence constraints, Journal of Operations Research, 45, 365–-377 (1997).
\bibitem{Miranda2018}
Miranda, P. L., Cordeau, J.F., Ferreira, D., Jans, R., Morabito, R., A decomposition heuristic for a rich production routing problem, Computers and Operations Research, 98, 211--230 (2018).
\bibitem{Molenbruch et al. 2017}
Molenbruch Y., Braekers K., Caris A., Schneide M., Typology and literature review for dial-a-ride problems, Journal of Expert Systems with Applications, 259, 295--325 (2017).
\bibitem{Mohring et al. 2004}
Mohring R.H., Skutella M., Stork F., Scheduling with AND/OR Precedence Constraints, SIAM Journal on Computing, 33, 393--415 (2004).
\bibitem{Moon et al. 2002}
Moon C., Kim J., Choi G., Seo, Y., An efficient genetic algorithm for the traveling salesman problem with precedence constraints, European Journal of Opertional Research, 140, 606–-617 (2002).
\bibitem{Moussavi2019}
Moussavi, S. E., Mahdjoub, M., Grunder, O., A matheuristic approach to the integration of worker assignment and vehicle routing problems: Application to home healthcare scheduling, Expert Systems with Applications, 125, 317-332 (2019).
\bibitem{Ozguven2010}
Ozguven, C., Ozbakir, L., Yavuz, Y., Mathematical models for job-shop scheduling problems with routing and process plan flexibility, Applied Mathematical Modelling, 34, 1539--1548 (2010).
\bibitem{Pezzella2008}
Pezzella, F., Morganti, G., Ciaschetti, G., A genetic algorithm for the flexible job-shop scheduling problem, Computer and Operational Research, 35, 3202–-3212 (2008),
\bibitem{Prot and Bellenguez-Morineau 2018}
Prot, D., Bellenguez-Morineau, O., How the structure of precedence constraints may change
the complexity class of scheduling problems,Journal of Scheduling, 21, 3--16 (2018).
\bibitem{Haddadene et al. 2016}
R-A-Haddadene S., Labadie N., Prodho C., Schneide M., A GRASP $\times$ ILS for the vehicle routing problem with time windows, synchronization and precedence constraints, Journal of Expert Systems with Applications, 66, 274--294 (2016).
\bibitem{Renaud et al. 2000}
Renaud, J., Boctor, F.F., Ouenniche, J., A heuristic for the pickup and delivery traveling salesman problem, Journal of Computers and Operations Research, 27, 905–-916 (2000).
\bibitem{Rivera et al. 2016}
Rivera, J. C., Afsar, H.M., Prins, C., Mathematical formulations and exact algorithm for the multitrip cumulative capacitated single-vehicle routing problem, European Journal of Operational Research, 249, 93--104 (2016)
\bibitem{Riise et al.2016}
Riise, A., Mannino, C., Lamorgese, L., Recursive logic-based Benders’ decomposition for multi-mode outpatient scheduling, European Journal of Operational Research, 255, 719--728 (2016).
\bibitem{Ropke2006}
Ropke, S., Pisinger, D., An adaptive large neighborhood search heuristic for the pickup and delivery problem with time windows, Journal of Transportation Science, 40, 455–-472 (2006).
\bibitem{Roshanaei et al.2017}
Roshanaei V., Luong C., M. Aleman D., Urbach, D., Propagating logic-based Benders’ decomposition approaches for distributed operating room scheduling, European Journal of Opeartional Research, 257, 439--455 (2017)
\bibitem{Sarin2005}
Sarin, S.C., Sherali, H.D., Bhootra, A., New tighter polynomial length formulations for the asymmetric traveling salesman problem with and without precedence constraints, Operations Research Letters, 33, 62--70 (2005).
\bibitem{Salman2016}
Salman, R., Algorithms for the Precedence Constrained Generalized Travelling Salesperson
Problem, Master’s thesis: Chalmers University of Technology. University of Gothenburg, (2016).
\bibitem{Sen2018}
Sen, A., Bulbul, K., Feillet, D., A survey on multi trip vehicle routing problem, International Logistics and Supply Chain Congress, (2018)
\bibitem{Shetty2008}
Shetty, V. K., Sudit, M., Nagi, R., Priority-based assignment and routing of a fleet of unmanned combat aerial vehicles, Computers and Operations Research, 35, 1813--1828 (2008).
\bibitem{Soares2019}
Soares, R., Marques, A., Amorim, P., Rasinmaki, J., Multiple vehicle synchronisation in a full truck-load pickup and delivery problem: A case-study in the biomass supply chain, European Journal of Operational Research, 277, 174--194 (2019).
\bibitem{Toth2002}
Toth, P., Vigo, D., The Vehicle Routing Problem: Society for Industrial and Applied Mathematics, (2002).
\bibitem{Tong2017}
Tong, L., Zhou, L., Liu, J., Zhou, X., Customized bus service design for jointly optimizing passenger-to-vehicle assignment and vehicle routing, Transportation Research Part C: Emerging Technologies, 85, 451--475 (2017).
\bibitem{Tran2012}
Tran, T.T., Beck, J.C., Logic-based benders decomposition for alternative resource scheduling with sequence dependent setups, In Proceedings of the 20th European Conference on Artificial Intelligence, 774--779 (2012).
\bibitem{Tzur2011}
Tzur, M., Drezner, E., A lookahead partitioning heuristic for a new assignment and scheduling problem in a distribution system, European Journal of Operational Research, 215, 325--336 (2011).
\bibitem{Unsal2019}
Unsal, O., Oguz, C., An exact algorithm for integrated planning of operations in dry bulk terminals, Transportation Research Part E: Logistics and Transportation Review, 126, 103--121 (2019).
\bibitem{Van Den Akker et al. 2005}
Van Den Akker J.M., Hoogeveen J.A., van Kempen J.W., Parallel machine scheduling through column generation: minimax objective functions, release dates, deadlines, and/or generalized precedence constraints, Lecture Notes in Computer Science, 4168, 648--659 (2005).
\bibitem{Yun2011}
Yun, Y., Moon, C., Genetic algorithm approach for precedence constrained sequencing problems, Journal of Intelligent Manufacturing, 22, 379–-388 (2011).
\bibitem{Zhang et al. 2020}
Zhang, A., Qi, X., Li, G., Machine scheduling with soft precedence constraints, European Journal of Operational Research, 282, 491--505 (2020).
\bibitem{Zulj et al. 2018}
Zulj I., H-Glock C., H-Grosse E., Schneide, M., Picker routing and storage-assignment strategies for precedence-constrained order picking, Journal of Computers and Industrial Engineering, 123, 338--347 (2018).
\end{thebibliography}
\end{document}